\documentclass[11pt]{amsart}
\usepackage{xcolor}
\usepackage[margin=1in]{geometry}
 \numberwithin{equation}{section}

\newcommand{\sk}{\ln\esk }
\newcommand{\des}{\Delta\rho }
\newcommand{\es}{\rho }
\newcommand{\sko}{ \ln\esko}

\newcommand{\esk}{\rho_k }
\newcommand{\esko}{\rho_{k-1} }
\newcommand{\esbj}{\rbj}
\newcommand{\eskx}{\rho_k(x)}
\newcommand{\uk}{u_k }

\newcommand{\uko}{u_{k-1} }

\newcommand{\utj}{\tilde{u}_j }
\newcommand{\wtj}{\tilde{w}_j }
\newcommand{\ctj}{\tilde{\sigma}_j }
\newcommand{\ubj}{\overline{u}_j }

\newcommand{\rbj}{\overline{\rho_j} }
\newcommand{\rtj}{\tilde{\rho_j} }

\newcommand{\gbj}{\overline{G_j }}
\newcommand{\ftj}{\tilde{F_j }}

\newcommand{\omt}{\Omega_T }
\newcommand{\io}{\int_\Omega }
\newcommand{\iot}{\int_{\Omega_T} }

\newcommand{\po}{\partial\Omega }

\newcommand{\pt}{\partial_t}
\newcommand{\edu}{\rho }

\newcommand{\ep}{\varepsilon}

\newtheorem{thm}{Theorem}[section]

\newtheorem{lem}[thm]{Lemma}

\begin{document}
	\title[A fourth order exponential PDE ]{
		Strong solutions to a fourth order exponential PDE describing epitaxial growth
		%Analysis of a continuum  model for broken bond crystal surfaces with evaporation and deposition effects
	}
	\author{ Brock C. Price and Xiangsheng Xu}\thanks
	%\addressa Multi-Dimensional Crystal Surface Model Analysis ofA Crystal Surface Model
	{%Liu's address: Department of Mathematics,  University, Durham, NC 27708. {\it Email:} jliu@phy.duke.edu.\\
%	Xu's address:	
Department of Mathematics and Statistics, Mississippi State
		University, Mississippi State, MS 39762.
		{\it Email}: bcp193@msstate.edu(Brock C. Price); xxu@math.msstate.edu(X. Xu). {\it J. Differential Equations, to appear.}}
	\begin{abstract} In this paper we prove the global existence of a strong solution to the initial boundary value problem for the exponential partial differential equation $\partial_tu-\Delta e^{-\Delta u}+e^{-\Delta u}-1=0$. The equation was proposed as a continuum model for epitaxial growth of crystal surfaces on vicinal surfaces with evaporation and deposition effects  \cite{GLLM}. Our investigations reveal that we must control the size of both  $\left\| e^{-\Delta u(x,0)}\right\|_{W^{2,2}(\Omega)}$ and $\left\| e^{\Delta u(x,0)}\right\|_{\infty,\Omega}$ suitably to achieve our results. Related results in 
		\cite{GM,LS} were established via the Weiner algebra framework. Here we offer a totally new approach, which seems to shed more light on the nature of exponential nonlinearity.
		% was employed.%That is, the two norms cannot be arbitrarily large. 
		%Our approach is totally different from the ones offered in%This
		%We offer a totally new perspective from which to view the problem than the ones 
	\end{abstract}
	\maketitle

	\section{Introduction}
	
\subsection{Problem background and statement of main results}	Let $\Omega$ be a bounded domain in $\mathbb{R}^N$ with $C^2$ boundary $\partial\Omega$. For each $T>0$, we consider the initial-boundary value problem
	\begin{eqnarray}
	\partial_t u &=& \Delta e^{-\Delta u}-e^{-\Delta u}+1  \ \ \ \mbox{in $\omt\equiv \Omega\times(0,T)$,}\label{p1}\\
	\nabla u\cdot\nu=\nabla e^{-\Delta u}\cdot\nu &=& 0 \ \ \ \mbox{ on $\Sigma_T\equiv \partial\Omega\times(0,T)$},\label{p2}\\
	u(x,0)&=& u_0(x) \ \ \ \mbox{on $\Omega$},\label{p3}
	\end{eqnarray}
	where $\nu$ is the unit outward normal vector to the boundary.
	
Equation (\ref{p1})	can be used to describe the evolution of a crystal surface \cite{GLLM}. 
%This problem arises in the study of crystal surfaces. 
In this case, $u$ is the surface height. The fourth order term in the equation represents the diffusion effect, while the lower order terms describe evaporation and deposition. Detailed information can be found in \cite{GLLM}.

Epitaxial growth is an important process in forming solid films and other nano-structures. Mathematical modeling of the process has attracted wide attentions \cite{GLLM}. Continuum models involving exponential nonlinearity were first derived in \cite{KDM} and more recently in \cite{MW,GLLM}. Mathematical analysis of such models in high space dimensions ($N\geq 2$) is very challenging due to the lack of estimates for the exponent term. It was first observed in \cite{LX} that one had to allow the possibility that the exponent be a measure-valued function. Later, the idea of ``exponential singularity'' was  employed in \cite{G1,GLL3,GLLX,X1,X5}. However, measure exponents do not arise in the one-dimensional case. See \cite{GLL2,GLLM}.

To remove the singularity in the exponent, the authors in \cite{GM,LS} introduced a rather sophisticated critical Wiener algebra space and showed that there existed a strong (no measure) solution as long as the norm of $u_0$ in the  Wiener algebra space was suitably small. The proof in  \cite{LS} employed the Fourier transform of the power series expansion of the exponential term. A similar approach was also adopted in 
\cite{GM}.
Here we offer a totally different perspective from which to view the problem. Our method is based upon Lemma \ref{prop2.2} below, a simple result first introduced in \cite{X4}. Denote by
$\|\cdot\|_{p,\Omega}$ the norm in the space $L^p(\Omega)$.  Our investigations reveal that we can obtain global existence of a strong solution by requiring the $W^{2,2}(\Omega)$ norm of $e^{-\Delta u_0}$ and $\|e^{\Delta u_0}\|_{\infty,\Omega}$ to be suitably small.
	
Before we state our main theorem, we give our definition of a strong solution.

\bigskip

\noindent {\bf Definition:} 
We say that a pair $(u,\rho)$ is a strong solution to (\ref{p1})-(\ref{p3}) if the following conditions hold:
\begin{enumerate}
	\item[(D1)]
	$u,\ \rho\in L^\infty(0,T;W^{2,2}(\Omega))\cap W^{1,2}(\Omega_T)$ with $\rho\geq c_0$ for some positive number $c_0$;
	%, $u\in C([0,T]; L^2(\Omega))\cap L^\infty(0,T; W^{2,2}(\Omega))$;
	%, where $\left(W^{1,2}(\Omega)\right)^*$ is the dual space of $W^{1,2}(\Omega)$ and $\mathcal{M} (\overline{\Omega_T})$ is the space of signed Radon measures on $\overline{\Omega_T}$;
	\item[(D2)] 
	%there hold the equations
	We have
	\begin{eqnarray}
		\frac{\partial u}{\partial t} &=& \Delta \rho-\rho+1  \ \ \ \mbox{a.e. on $\omt$,}\nonumber\\
		-\Delta u&=&\ln \rho \ \ \ \mbox{a.e. on $\omt$,}\nonumber\\
		\nabla u\cdot\nu=	\nabla \rho\cdot\nu &=& 0 \ \ \ \mbox{a.e. on $\Sigma_T$}.\nonumber
	\end{eqnarray}
	%	where $g_a$ is the absolutely continuous part of $-\Delta u$ with respect to
	%	the Lebesgue measure in the Lebesgue Decomposition Theorem;
	The initial condition (\ref{p3}) is satisfied in the space $C([0,T]; L^2(\Omega))$.
\end{enumerate}

\bigskip

Our main result is the following 
\begin{thm}[Main Theorem]\label{thm1}	Assume:
	\begin{enumerate}
		\item[\textup{(H1)}] $\Omega$ is a bounded domain in $\mathbb{R}^N$ with $C^2$ boundary;
		\item[\textup{(H2)}] $N=2$ or $3$;
		\item[\textup{(H3)}] $u_0 \in W^{2,2}(\Omega)$ is such that $ e^{-\Delta u_0 }\in W^{2,2}(\Omega) $. Moreover, we have
		the consistence conditions
		\begin{equation}\label{bcon}
			\nabla u_0\cdot\nu=0,\ \ \nabla e^{-\Delta u_0 }\cdot\nu=0\ \ \mbox{a.e. on $\po$.}
		\end{equation}
	\end{enumerate}  Then there exist two positive numbers $s_0, s_1$ determined by $\Omega$ only
% a positive number $c_0=c_0(N, \Omega)$ 
such that 
	problem (\ref{p1})-(\ref{p3}) has a global strong solution whenever 
	\begin{equation}\label{ez}
		\varepsilon_0\equiv	\left\| e^{-\Delta u_0}\right\|_{W^{2,2}(\Omega)}<s_0, \ \ \left\|e^{\Delta u_0}\right\|_{\infty, \Omega}< s_1.
	\end{equation}
\end{thm}

By a global strong solution, we mean that for each $T>0$ there is a strong solution $u$ to \eqref{p1}-\eqref{p3} on $\Omega_T$. We can infer from (H1), (H2), and the Sobolev embedding theorem that
%, for each $p>\frac{3}{2}$ there is a positive number $c=c(\Omega, p)$ such that
\begin{equation}\label{emb}
	\|u\|_{\infty,\Omega}\leq c(\Omega, p)\|u\|_{W^{2,p}(\Omega)}\ \ \mbox{for each $u\in W^{2,p}(\Omega)$ whenever $p>\frac{3}{2}$.}
\end{equation}
Hence we also have $ e^{-\Delta u_0}\in L^{\infty} (\Omega)$. Note that the two inequalities in \eqref{ez} are not contradictory. Roughly speaking, the first one controls the set where $-\Delta u_0$ is very large, while the second one is concerned with the set where $\Delta u_0$ is very large. In fact, our assumptions here reveal the true nature of the exponential nonlinearity. That is, the composite function $e^{-\Delta u_0}$ can still behave well even if the exponent term $-\Delta u_0$ displays singularity near the set $\{-\Delta u_0=-\infty\}$. The second inequality in \eqref{ez} is assumed to prevent this from happening.  We refer the reader to \cite{LX} for more discussions in this regard.
 % Therefore, the  seems to be necessary.
% In this sense, our weak solution is global.\in W^{1,2}(\Omega)\cap \cap L^\infty(\Omega)

\subsection{A priori estimates for smooth solutions}\label{sub2} To gain some insights into our problem, we proceed to perform some formal analysis. By ``formal'', we mean that the solution $u$ to \eqref{p1}-\eqref{p3} is as smooth as we desire
%possesses enough regularity properties 
so that all the subsequent calculations in this subsection make sense. However, the essence of our approach is already demonstrated here.

To simplify our presentation, we introduce the functions
\begin{equation}\label{ur}
	\rho=e^{-\Delta u}, \ \ G=	\partial_tu+\rho-1.
\end{equation}Then \eqref{p1} becomes
\begin{eqnarray}
	-\Delta \rho+G&=&0\ \ \mbox{in $\Omega_T$.}\label{s1}
%	-\Delta u&=&\ln \rho\ \ \mbox{in $\Omega$.}
\end{eqnarray}
Square both sides of this equation and then integrate it with respect to $x$ over $\Omega$ to get
\begin{eqnarray}
	\io\left[\left(\Delta \rho\right)^2+G^2\right]dx-2\io G\Delta \rho dx=0.\label{ff1}
\end{eqnarray}
Note from \eqref{ur} that
\begin{eqnarray}
-2\io G\Delta \rho dx&=&	-2\io\partial_tu\Delta \rho dx+2\io\left|\nabla\edu\right|^2dx\nonumber\\
&=& -2\int_{\Omega}e^{-\Delta u}\partial_t\Delta u \,   \,dx+2\io\left|\nabla\edu\right|^2dx\nonumber\\
&=&2\frac{d}{dt}\int_{\Omega}\rho \,dx+2\io\left|\nabla\edu\right|^2dx.\nonumber
\end{eqnarray}
%$$=  =-.$$
Substitute this into \eqref{ff1} and integrate the resulting equation to obtain
\begin{eqnarray}
	\sup_{0\leq t\leq T}2\io\rho(x,t)dx+
	%\label{ff3}\\\leq\varepsilon_0
	\iot\left[G^2+\left(\Delta \rho\right)^2+2\left|\nabla\edu\right|^2\right]dxdt+
\leq 4\io e^{-\Delta u_0(x)}dx\leq 4\varepsilon_0,\label{nm1}
%2\int_{\Omega} e^{-\Delta u_0(x)}\,dx+|\Omega_T|.
%\label{f2}
\end{eqnarray}
where $\varepsilon_0$ is given as in \eqref{ez}.

Next we differentiate (\ref{s1}) with respect to $t$ and then use $G$ as a test function in the resulting equation to obtain
	\begin{eqnarray}
-\int_{\Omega}\partial_t\Delta\edu \, G\, dx+\frac{1}{2}\frac{d}{dt}\int_{\Omega}G^2\, dx=0.\label{r12}
%&=&-2\io\pt\rho\pt udx\nonumber\\
%&=&-2\io\sqrt{\rho}\pt\sqrt{\rho}\pt udx\nonumber\\
%&\leq&\io\left(\pt\sqrt{\rho}\right)^2dx+\io\rho\left(\pt u\right)^2dx.
%&=&	\int_{\Omega}\partial_t\Delta\edu \, \partial_t u\, dx-\int_{\Omega}\partial_t\edu \, \partial_t u\, dx\nonumber\\
%&=&-4\int_{\Omega}|\partial_t e^{-\frac{\Delta u}{2}}|^2\, dx
	\end{eqnarray}
Observe that 
\begin{eqnarray}
\int_{\Omega}\partial_t\Delta\edu \, G\, dx&=&	\int_{\Omega}\partial_t\Delta \rho\, \partial_t u\, dx+\int_{\Omega}\partial_t\Delta \rho\, (\rho-1)\, dx\nonumber\\
	&=&\int_{\Omega}\partial_te^{-\Delta u}\, \partial_t\Delta u\, dx-\io\pt\nabla\rho\cdot\nabla\rho dx\nonumber\\
	&=&-\int_{\Omega}e^{-\Delta u} \, |\partial_t\Delta u|^2 \, dx-\frac{1}{2}\frac{d}{dt}\io|\nabla\rho|^2dx\nonumber\\
	&=&-4\int_{\Omega}|\partial_t \sqrt{\rho}|^2\, dx-\frac{1}{2}\frac{d}{dt}\io|\nabla\rho|^2dx.\nonumber
\end{eqnarray}
Substitute this into \eqref{r12}
to derive
\begin{eqnarray}
%	\frac{1}{2}\frac{d}{dt}\int_{\Omega}(\partial_t u)^2\, dx\leq 
\lefteqn{\sup_{0\leq t\leq T}\int_{\Omega}\left(\frac{1}{2}|\nabla\rho|^2+\frac{1}{2}G^2	\right)\, dx+4\int_{\Omega_T}|\partial_t \sqrt{\rho}|^2\, dx}\nonumber\\
&\leq& \io\left(|\nabla\rho(x,0)|^2+G^2(x,0)\right)dx\nonumber\\
&=&\io\left|\nabla e^{-\Delta u_0(x)}\right|^2dx+\io\left|\Delta e^{-\Delta u_0(x)}\right|^2dx\leq \varepsilon_0^2.\label{nm3}
%\nonumber\\
%&=&\io\left[\left(\Delta e^{-\Delta u_0(x)} \right)^2+\left|\nabla e^{-\Delta u_0(x)}\right|^2\right]dx.
%\io\left(2\pt\sqrt{\rho}+\frac{1}{2}\sqrt{\rho}\pt u\right)^2dx=\frac{1}{4}	\io\rho\left(\pt u\right)^2dx.
\end{eqnarray}
By virtue of (H2) and Lemma \ref{wbd} below, there is a positive number $c=c(\Omega)$ such that
\begin{equation}\label{rub}
	\|\rho(\cdot,t)\|_{\infty,\Omega}\leq c\|\rho(\cdot,t)\|_{1,\Omega}+c\|G(\cdot,t)\|_{2,\Omega}\leq c\varepsilon_0.
\end{equation}
%This together with \eqref{r23} implies\| e^{-\Delta u_0}\|_{W^{2,2}(\Omega)}
%\begin{equation}
%\|\rho\|_{\infty,\Omega_T}\leq c\varepsilon_0.
%\end{equation}
%Subsequently,
%\begin{equation}
%	\sup_{0\leq t\leq T}\int_{\Omega}|\partial_t u(x,t)|^2\,dx\leq 2	\sup_{0\leq t\leq T}\int_{\Omega}G^2\,dx+2\sup_{0\leq t\leq T}\io(\rho-1)^2dx\leq c\varepsilon_0^2+c.
%\end{equation}

It follows from \eqref{ur} that
\begin{equation}
	-\Delta u=\ln \rho\ \ \mbox{in $\Omega$.}\nonumber
\end{equation}Integrate this equation over $\Omega$ and use \eqref{p2} to obtain
	$$\int_{\Omega}\ln\rho  \, dx=0.$$
	Keeping this and (\ref{nm1}) in mind, we estimate
	\begin{eqnarray}
	\int_{\Omega}|\ln\rho | \, dx &=& \int_{\Omega}\ln^+\rho  \, dx
	+\int_{\Omega}\ln^-\rho  \, dx\nonumber\\
	&=&2\int_{\Omega}\ln^+\rho \,dx
	-\int_{\Omega}\ln\rho  \, dx\nonumber\\
	&\leq& 2\int_{\Omega}\rho  \, dx
	\leq 2\io e^{-\Delta u_0(x)}dx\leq 2\varepsilon_0.\nonumber
	\end{eqnarray}
Fix $L>1$. We have
\begin{equation}\label{rsm}
	\left|\left\{\rho\leq \frac{1}{L}\right\}\right|\leq \frac{1}{\ln L}\int_{\{\rho\leq \frac{1}{L}\}}|\ln\rho|dx\leq\frac{2\varepsilon_0}{\ln L}.
\end{equation}

Set 
\begin{equation}
	w=\frac{1}{\rho}.\nonumber
\end{equation}
We easily verify that
\begin{eqnarray}
	%	\rho&=& w^{-\alpha},\\
	\Delta \rho&=&- w^{-2}\Delta w+2w^{-3}|\nabla w|^2.\nonumber
\end{eqnarray}
Subsequently, $w$ satisfies the boundary value problem
\begin{eqnarray}
	- \Delta w+2w^{-1}|\nabla w|^2&=&Gw^2\ \ \mbox{in $\Omega$},\nonumber\\
	\nabla w\cdot\nu&=&0\ \ \mbox{on $\po$.}\nonumber
\end{eqnarray}
We can infer from Lemma \ref{wbd} below and (H2) that there is a positive number $c=c(\Omega)$ such that
\begin{eqnarray}
	\|w(\cdot,t)\|_{\infty,\Omega}&\leq& c\|w(\cdot,t)\|_{1,\Omega}+c\|G(\cdot,t)w^2(\cdot,t)\|_{2, \Omega}\nonumber\\
	&\leq &c\int_{\{w\leq L\}}wdx+c\int_{\{w> L\}}wdx+c	\|w(\cdot,t)\|_{\infty,\Omega}^2\|G(\cdot,t)\|_{2, \Omega}\nonumber\\
	&\leq &cL+c	\|w(\cdot,t)\|_{\infty,\Omega}\left|\left\{\rho\leq \frac{1}{L}\right\}\right|+c\varepsilon_0	\|w(\cdot,t)\|_{\infty,\Omega}^2\nonumber\\
	&\leq & cL+\frac{c\varepsilon_0\|w(\cdot,t)\|_{\infty,\Omega}}{\ln L}+c\varepsilon_0\|w(\cdot,t)\|_{\infty,\Omega}^2.
		\label{xu4}
\end{eqnarray}
Here we have used \eqref{nm3} and \eqref{rsm}.
 Consider the quadratic function 
$$Q(s)=c\varepsilon_0 s^2-\left(1-\frac{c\varepsilon_0}{\ln L}\right)s+cL\ \ \mbox{on $(0, \infty)$}.$$
%\subsection{Results}
Then \eqref{xu4} says
$$Q(\|w(\cdot,t)\|_{\infty,\Omega})\geq 0\ \ \mbox{for each $t\in [0,T]$.}$$
Suppose that $\|w(\cdot,t)\|_{\infty,\Omega}$ is a continuous function of $t$.  According to the proof of Lemma \ref{prop2.2} below, if we choose $L>1$ and $\varepsilon_0$  so that
\begin{equation}\label{xu1}
1-\frac{c\varepsilon_0}{\ln L}>0,\ \	\left(1-\frac{c\varepsilon_0}{\ln L}\right)^2> 4c^2L\varepsilon_0,
\end{equation}
then 
\begin{equation}
		\|w(\cdot,t)\|_{\infty,\Omega}\leq\frac{1-\frac{c\varepsilon_0}{\ln L}-\sqrt{\left(1-\frac{c\varepsilon_0}{\ln L}\right)^2- 4c^2L\varepsilon_0}}{2c\varepsilon_0}\equiv g(\varepsilon_0, L)\ \ \mbox{for $t>0$}\nonumber
\end{equation}
whenever
\begin{equation}
	\|w(\cdot,0)\|_{\infty,\Omega}\leq g(\varepsilon_0, L).\nonumber
	%\frac{1-\frac{c\varepsilon_0}{\ln L}-\sqrt{\left(1-\frac{c\varepsilon_0}{\ln L}\right)^2- 4c^2L\varepsilon_0}}{2c\varepsilon_0}.
\end{equation}
Take the square root of the second inequality in \eqref{xu1} to derive
$$-\frac{c}{\ln L}\varepsilon_0-2c\sqrt{L}\sqrt{\varepsilon_0}+1>0.$$
Solving this inequality yields
\begin{equation}\label{xu5}
	\sqrt{\varepsilon_0}<\sqrt{L\ln^2L+\frac{\ln L}{c}}-\sqrt{L}\ln L\equiv h(L).
\end{equation}
By (6) in Lemma \ref{l24} below,
$$h(L)\leq \sqrt{\frac{\ln L}{c}}.$$
That is, \eqref{xu5} implies the first inequality in \eqref{xu1}.
We easily see that
$$h(1)=0,\ \ \lim_{L\rightarrow\infty}h(L)=0.$$
Thus $h(L)$ attains its maximum value at some point $L_0\in (1,\infty)$. We take 
\begin{equation}
s_0= h^2(L_0).	\nonumber
\end{equation}
To determine $s_1$, it is easy to see that
\begin{eqnarray}
	g(\varepsilon_0,L)&=&\frac{ 2cL}{1-\frac{c\varepsilon_0}{\ln L}+\sqrt{\left(1-\frac{c\varepsilon_0}{\ln L}\right)^2- 4c^2L\varepsilon_0}}\nonumber\\
	&=&\frac{ 2cL}{1-\frac{c\varepsilon_0}{\ln L}+\sqrt{\frac{c^2}{\ln^2 L}\varepsilon_0^2- \left(\frac{2c}{\ln L}+4c^2L\right)\varepsilon_0+1}},\nonumber
\end{eqnarray}
which is an increasing function of $\varepsilon_0$ on the interval $(0, s_0)$. Thus we take
\begin{equation}
 s_1=g(0, L_0)=cL_0.\nonumber
\end{equation}
Whenever $\|e^{-\Delta u_0}\|_{W^{2,2}(\Omega)}<s_0,\  \|e^{\Delta u_0}\|_{\infty,\Omega}< s_1$, we have
\begin{equation}
	\|e^{\Delta u(\cdot,t)}\|_{\infty,\Omega}\leq g(\|e^{-\Delta u_0}\|_{W^{2,2}(\Omega)},L_0)\ \ \mbox{for all $t>0$.}\nonumber\\
\end{equation}
This together with \eqref{rub} implies
\begin{equation}
\Delta u\in L^\infty(\Omega).\nonumber\\
\end{equation}
In particular, the exponent term is not a measure.
%\begin{eqnarray}
%	s_1= \left.\frac{1-\frac{c\varepsilon_0}{\ln L}-\sqrt{\left(1-\frac{c\varepsilon_0}{\ln L}\right)^2- 4c^2L\varepsilon_0}}{2c\varepsilon_0}\right|_{\varepsilon_0=s_0, L=L_0}.
%\end{eqnarray}
% Since $\sqrt{s}$ is a concave function on $(0,\infty)$, we have

	A solution to (\ref{p1})-(\ref{p3}) will be constructed as the limit of a sequence of approximate solutions. The key is to design an approximation scheme so that all the calculations in Subsection \ref{sub2} can be justified. This is accomplished in Sections 2 and 3. To be more specific,
	%  This paper is organized as follows. 
	in Section 2 we state a few preparatory lemmas and present our approximate 
	problems. The existence of a classical solution is established for these problems. We form a sequence of approximate solutions based upon implicit discretization in the time variable. Section 3 is devoted to the proof of the discretized versions of the results in Subsection \ref{sub2}. These estimates are enough to justify passing to the limit.

	\section{Approximate Problems}
	\setcounter{equation}{0}
	Before we present our approximate problems, we state a few preparatory lemmas.
	
	\begin{lem}
		Let $\Omega$ be a bounded domain in $\mathbb{R}^N$.\begin{enumerate}
			\item[(i)] If $\Omega$ is convex, then
			\begin{equation}
			\int_\Omega(\Delta u)^2 \, dx\geq\int_\Omega|\nabla^2 u|^2 \,dx \ \ \mbox{
				for all $u\in W^{2,2}(\Omega)$ with $\nabla u\cdot\nu=0$ on $\po$.}\nonumber
			\end{equation}
			\item[(ii)] If $\po$ is $C^2$, then there is a positive constant $c$
			depending only on $N, \Omega$ and the smoothness of the boundary
			such that
			\begin{equation}\label{2.2}
			\int_\Omega(\Delta u)^2dx+\int_\Omega|\nabla
			u|^2dx\geq
			c\int_\Omega|\nabla^2u|^2
			dx \ \
			\end{equation}
			for  all $u\in W^{2,2}(\Omega)$ with $\nabla u\cdot\nu=0$ on
			$\po$.
		%	\item[(iii)] If $\po$ is Lipschitz, then there exists a positive number $c=c(N)$ with the property
		%	\begin{equation}
		%	\int_\Omega\left|\nabla\sqrt{u}\right|^4dx\leq c\int_\Omega|\nabla^2u|^2dx
		%	\end{equation}
		%	for all $u\in W^{2,2}(\Omega)$ with $u\geq 0$ on $\Omega$ and $\nabla u\cdot\nu=0$ on
		%	$\po$.
		\end{enumerate}
	\end{lem}
	We refer the reader to \cite{X3} for some background information on this lemma. 
	%We only mention that (iii) first appeared in \cite{GST}.
	
		Our existence theorem is based upon the following fixed point theorem, which is often called the Leray-Schauder Theorem (\cite{GT}, p.280).
		\begin{lem}
			Let $B$ be a map from a Banach space $\mathcal{B}$ into itself. Assume:
			\begin{enumerate}
				\item[\textup{(LS1)}] $B$ is continuous;
				\item[\textup{(LS2)}] the images of bounded sets of $B$ are precompact;
				\item[\textup{(LS3)}] there exists a constant $c$ such that
				$$\|z\|_{\mathcal{B}}\leq c$$
				for all $z\in\mathcal{B}$ and $\sigma\in[0,1]$ satisfying $z=\sigma B(z)$.
			\end{enumerate}
			Then $B$ has a fixed point.
		\end{lem}
	
	Relevant interpolation inequalities for Sobolev spaces are listed in the following lemma.
	\begin{lem}\label{l23}	Let $\Omega$ be a bounded domain in $\mathbb{R}^N$. Then we have:
		\begin{enumerate}
			\item $ \|f\|_{q,\Omega}\leq\varepsilon\|f\|_{r,\Omega}+\varepsilon^\sigma \|f\|_{p,\Omega}$, where
			$\varepsilon>0, p\leq q< r$, and $\sigma=\left(\frac{1}{p}-\frac{1}{q}\right)/\left(\frac{1}{q}-\frac{1}{r}\right)$;
			\item If $\po$ is $C^2$, for each $\varepsilon >0$ and each $p\in [2, 2^*)$, where $2^*=\frac{2N}{N-2}$ if $N>2$ and any number bigger than $2$ if $N=2$, there is a positive number $c=c(\varepsilon, p)$ such that
			\begin{eqnarray}
			\|f\|_{p,\Omega}&\leq &\varepsilon\|\nabla f\|_{2,\Omega}+c\|f\|_{1,\Omega}\ \ \mbox{for all $f\in
				W^{1,2}(\Omega)$},\nonumber\\	
			\|\nabla g\|_{p,\Omega}&\leq &\varepsilon\|\nabla^2 g\|_{2,\Omega}+c\|g\|_{1,\Omega}\ \ \mbox{for all $g\in
				W^{2,2}(\Omega)$}.	\nonumber
			\end{eqnarray}
		\end{enumerate}
	\end{lem}

Finally, we collect a few frequently used elementary inequalities in the following lemma.
	\begin{lem}\label{l24} For $x,y\in \mathbb{R}^N$, $s,t\in \mathbb{R}$, and $\ a, b\in (0,\infty)$, we have:
		\begin{enumerate}
			\item[(3)] $x\cdot(x-y)\geq \frac{1}{2}(|x|^2-|y|^2);$
			\item[(4)] if $f$ is an increasing function on $\mathbb{R}$ and $F$ an anti-derivative of $f$, then
			$$f(s)(s-t)\geq F(s)-F(t).
			%\ \ \mbox{ for all $s,t\in \mathbb{R}$}.
			$$
			In particular, there hold the inequalities
			%we have
			\begin{eqnarray}
			a(\ln a-\ln b)&\geq& a-b\  \ \ \mbox{and}\label{ota1}\\
			 (a-b)\ln a&\geq&a\ln a-b\ln b-(\ln a-\ln b) ;
			%\ \ \mbox{ for all $s,t\in \mathbb{R}$.}
			\label{om1}
			\end{eqnarray}
			\item[(5)] we have
			%There holds
			\begin{equation}
			(a-b)(\ln a-\ln b)\geq 2\left(\sqrt{a}-\sqrt{b}\right)^2;
			%\ \mbox{ for all $s,t\in \mathbb{R}$.}
			\label{om2}
			\end{equation}
			\item[(6)] there hold
			%Let $a, b\in \mathbb{R}^+$. Then we have
			\begin{eqnarray*}
			(a+b)^\alpha&\leq&a^\alpha+b^\alpha\ \ \mbox{if $0<\alpha\leq 1$},\\
			(a+b)^\alpha&\leq&2^{\alpha-1}(a^\alpha+b^\alpha)\ \ \mbox{if $\alpha>1$},\\
			ab&\leq& \varepsilon a^p+\frac{1}{\varepsilon^{q/p}}b^q \ 
			\mbox{if $\varepsilon>0,\, p, \, q>1$ with $\frac{1}{p}+\frac{1}{q}=1$}.
			\end{eqnarray*}
		\end{enumerate}
	\end{lem}
The proof of the lemma is also rather elementary. We refer the reader to \cite{LX} for details.
 \begin{lem}\label{ynb}
	Let $\{y_n\}, n=0,1,2,\cdots$, be a sequence of positive numbers satisfying the recursive inequalities
	\begin{equation}
		y_{n+1}\leq cb^ny_n^{1+\alpha}\ \ \mbox{for some $b>1, c, \alpha\in (0,\infty)$.}\nonumber
	\end{equation}
	If
	\begin{equation*}
		y_0\leq c^{-\frac{1}{\alpha}}b^{-\frac{1}{\alpha^2}},
	\end{equation*}
	then $\lim_{n\rightarrow\infty}y_n=0$.
\end{lem}
This lemma can be found in (\cite{D}, p.12).
\begin{lem}\label{wbd}Let $w\in W^{1,2}(\Omega)$ be a weak solution of the boundary value problem
	\begin{eqnarray}
		-\Delta w&=&f\ \ \mbox{in $\Omega$,}\label{rec1}\\
		\nabla w\cdot\nu&=&0\ \ \mbox{on $\po$.}\label{rec2}
	\end{eqnarray}
Then for each $p>\frac{N}{2}$ there is a positive number $c=c(N, p, \Omega)$ such that
\begin{equation}
	\|w\|_{\infty,\Omega}\leq c\|w\|_{1,\Omega}+c\|f\|_{p,\Omega}.\nonumber
\end{equation}
\end{lem}
This result is well known. Since the proof is rather simple, we reproduce it here.
\begin{proof}Without loss of generality, assume that
	$$\|w\|_{\infty,\Omega}=\|w^+\|_{\infty,\Omega}.$$
We will only focus on the case $$N=2.$$	Let $p$ be given as in the lemma, and
let 
\begin{equation}\label{kcon}
	K>\|f\|_{p,\Omega}
\end{equation} be selected as below. Set
$$k_n=K-\frac{K}{2^{n}},\ \ y_n=\frac{1}{|\Omega|}\io\left(w-k_n\right)^+dx,\ \ n=0,1,2,\cdots.$$
Use $(w-k_{n+1})^+-y_{n+1}$ as a test function in \eqref{rec1} to get
\begin{eqnarray}
	\io\left|\nabla(w-k_{n+1})^+\right|^2dx&=&\io f\left[(w-k_{n+1})^+-y_{n+1}\right]dx.\label{key10}
\end{eqnarray}
For each $s>1$ we obtain from the Sobolev inequality that
\begin{eqnarray}
\io f\left[(w-k_{n+1})^+-y_{n+1}\right]dx	&\leq&\|f\|_{\frac{s}{s-1},\{w\geq k_{n+1}\}}\|(w-k_{n+1})^+-y_{n+1}\|_{s,\Omega}\nonumber\\
	&\leq &cK\left|\{w\geq k_{n+1}\}\right|^{\frac{s-1}{s}-\frac{1}{p}}\|\nabla(w-k_{n+1})^+\|_{\frac{2s}{2+s},\Omega}\nonumber\\
	&\leq &cK\left|\{w\geq k_{n+1}\}\right|^{1-\frac{1}{p}}\|\nabla(w-k_{n+1})^+\|_{2,\Omega}.\nonumber
	%	&\equiv&\frac{2^{n+2}}{\alpha}D(K)y_n^{\frac{1}{p}}.
	%\left(\int_{\{w>\frac{K}{2}\}}(1+|f|)^{\frac{p}{p-1}}dx\right)^{\frac{p-1}{p}}\|1+f\|_{\frac{p}{p-1},\{w>\frac{K}{2}\}}	&\leq&\frac{1}{4}y_n^{\frac{1}{2p}}\left|\{w\geq k_{n+1}\}\right|^{1-\frac{1}{2p}}+2^{n+2}\|w\|_{\infty, \Omega}\left(\int_{\{w>\frac{K}{2}\}}|\pt u|^{\frac{p}{p-1}}dx\right)^{\frac{p-1}{p}}y_n^{\frac{1}{p}}\nonumber\\
	%	&\leq&\frac{2^{(n+2)(2p-1)}}{4K^{2p-1}}y_n+2^{n+2}\|w\|_{\infty, \Omega}\|\pt u\|_{2,\Omega}\left|\left\{w>\frac{K}{2}\right\}\right|^{\frac{p-1}{p}-\frac{1}{2}}y_n^{\frac{1}{p}}.y_n
\end{eqnarray}%It immediately follows that
Use this in \eqref{key10} to get
\begin{equation}
	\|\nabla(w-k_{n+1})^+\|_{2,\Omega}\leq cK\left|\{w\geq k_{n+1}\}\right|^{1-\frac{1}{p}}.\nonumber
\end{equation}
With this and the Sobolev inequality in mind, we deduce that
\begin{eqnarray}
	y_{n+1}&\leq&\left(\io\left[\left(w-k_{n+1}\right)^+\right]^{s}dx\right)^{\frac{1}{s}}\left|\{w\geq k_{n+1}\}\right|^{1-\frac{1}{s}}\nonumber\\
	&\leq&c	\left(\left(\io\left|\nabla\left(w-k_{n+1}\right)^+\right|^{\frac{2s}{s+2}}dx\right)^{\frac{s+2}{2s}}+\io\left(w-k_{n+1}\right)^+dx\right)\left|\{w\geq k_{n+1}\}\right|^{\frac{s-1}{s}}\nonumber\\
	&\leq&c	\left(\left\|\nabla\left(w-k_{n+1}\right)^+\right\|_{2,\Omega}\left|\{w\geq k_{n+1}\}\right|^{\frac{1}{s}}+\io\left(w-k_{n+1}\right)^+dx\right)\left|\{w\geq k_{n+1}\}\right|^{\frac{s-1}{s}}\nonumber\\
	&\leq&c\left(K \left|\{w\geq k_{n+1}\}\right|^{1+\frac{1}{s}-\frac{1}{p}}+y_n\right)\left|\{w\geq k_{n+1}\}\right|^{\frac{s-1}{s}}.\label{key3}
	%\nonumber\\
%	&\leq&c4^{n}\|1+|f|\|_{2,\Omega}^2y_n\left|\{w\geq k_{n+1}\}\right|^{\frac{N+2}{2N}}+cy_n\left|\{w\geq k_{n+1}\}\right|^{\frac{N+2}{2N}}\nonumber\\
%	&\leq &\frac{cb^n\left(\|1+|f|\|_{2,\Omega}^2+1\right)}{K^{\frac{4}{3}}}y_n^{1+\frac{N+2}{2N}},\ \ \  \  b=4 (16)^{\frac{N+2}{2N}}>1.
	%&\leq&\frac{c4^{\frac{p(n+2)(N-(N-2)p)}{N}}}{K^{\frac{2p(N-(N-2)p)}{N}}}y_n^{1-\frac{(N-2)p}{N}}\nonumber\\
	%	&\leq& \frac{cb^n}{K^{\frac{2p(N-(N-2)p)}{N}}}
\end{eqnarray}
We easily see that
\begin{equation}
	y_n\geq\frac{1}{|\Omega|}\int_{\{w\geq k_{n+1}\}}\left(w-k_n\right)^+dx\geq \frac{K}{2^{n+1}|\Omega|}\left|\{w\geq k_{n+1}\}\right|.\nonumber
\end{equation}
Take
$$\alpha=\min\left\{1-\frac{1}{p},\frac{s-1}{s} \right\}.$$ 
Our assumption on $p$ implies 
$$\alpha>0.$$
We can obtain from \eqref{key3} that
\begin{eqnarray}
	y_{n+1}&\leq& cK\left|\{w\geq k_{n+1}\}\right|^{1+\alpha}+cy_n\left|\{w\geq k_{n+1}\}\right|^\alpha\nonumber\\
	&\leq&\frac{c2^{(1+\alpha)n}}{K^\alpha}y_n^{1+\alpha}.\nonumber
\end{eqnarray}
%A slight adjustment is needed if $N=2$.
 According to Lemma \ref{ynb}, if we choose $K$ so large that
$$y_0= \frac{1}{|\Omega|}\io w^+dx\leq cK,$$
then $w\leq K$. In view of\eqref{kcon}, we conclude
%This results in
\begin{eqnarray}
	w&\leq& c	\|w\|_{1,\Omega}+c	\|f\|_{p,\Omega}.\nonumber
\end{eqnarray}
\end{proof}
\begin{lem}\label{prop2.2}
	Let $h(\tau)$ be a continuous non-negative function defined on $[0, T_0]$ for some $T_0>0$. Suppose that there exist three positive numbers $\ep, \delta, b $ such that
	\begin{equation}\label{f22}
		h(\tau)\leq \ep h^{1+\delta}(\tau)+b\ \ \mbox{for each $\tau \in[0, T_0]$}.
	\end{equation}
	%	where  are all positive numbers.
	Then
	\begin{equation}
		h(\tau)\leq \frac{1}{[\ep(1+\delta)]^{\frac{1}{\delta}}}\equiv s_0
		%\frac{(b+\delta)(1+\delta)}{\delta}
		\ \ \mbox{for each $\tau\in [0, T_0]$},
	\end{equation}  provided that
	\begin{equation}\label{f1}
		\ep\leq \frac{\delta^\delta}{(b+\delta)^\delta(1+\delta)^{1+\delta}}\ \ \mbox{and} \ \ h(0)\leq
		%\frac{(b+\delta)(1+\delta)}{\delta}.
		s_0.
	\end{equation}
\end{lem}
The proof is given in \cite{X4}. For the convenience of the reader, we will reproduce it here.
\begin{proof}  Consider the function $f(s)=\ep s^{1+\delta}-s+b$ on $[0,\infty)$. Then condition \eqref{f22} simply says
	\begin{equation}\label{f3}
		f(h(\tau))\geq 0\ \ \mbox{for each $\tau\in [0,T_0]$.}
	\end{equation} 
	It is easy to check that the function $f$ achieves its minimum value 
	at $s_0=\frac{1}{[\ep(1+\delta)]^{\frac{1}{\delta}}}$. The minimum value
	\begin{eqnarray}
		f(s_0)&=&\frac{\ep}{[\ep(1+\delta)]^{\frac{1+\delta}{\delta}}}-\frac{1}{[\ep(1+\delta)]^{\frac{1}{\delta}}}+b\nonumber\\
		&=&b-\frac{\delta}{\ep^{\frac{1}{\delta}}(1+\delta)^{\frac{1+\delta}{\delta}}}.\nonumber
	\end{eqnarray}	
	By the first inequality in \eqref{f1}, $f(s_0)	\leq -\delta$. Consequently, the equation $f(s)=0$ has exactly two solutions $0<s_1< s_2$ with $s_0$ lying in between.
	%\in (s_1,s_2)$. 
	Evidently, $f$ is positive on $[0, s_1)$, negative on $(s_1,s_2)$, and positive again on $(s_2, \infty)$.
	%\begin{equation}
	%f(s)\left\{\begin{array}{cc}
	%>0&\mbox{if $},\\
	%<0&\mbox{if $s\in (s_1,s_2)$,}\\
	%>0&\mbox{if $s\in .}
	%\end{array}\right.Obviously, 	 
	%\end{equation}
	The range of $h$ is a closed interval because of its continuity, and this interval is either contained in $ [0, s_1)$ or $(s_2, \infty)$ due to \eqref{f3}. The latter cannot occur due to the second inequality in \eqref{f1}. Thus the lemma follows.
\end{proof}

%Now we are ready to introduce our approximate problems.
We largely follow \cite{X1}	for the construction of approximate problems. For this purpose, let
	\begin{equation}
	\tau>0 \ \ \mbox{and $v\in L^2(\Omega)$.} \label{ot3}
	 \end{equation}  Consider the boundary value problem
	\begin{eqnarray}
	-\des+\rho+\tau\ln\rho  &=&-\frac{u-v}{\tau}+1 \ \ \ \mbox{in $\Omega$,}\label{ot1}\\
	-\Delta u+\tau u &=&\ln\rho  \ \ \ \mbox{in $\Omega$},\\
	\nabla u\cdot\nu=\nabla\es\cdot\nu&=& 0\ \ \ \mbox{on $\partial\Omega$}.\label{ot2}
	\end{eqnarray}
This problem will serve as a basis for our approximation. To obtain an existence assertion for this problem, we first need to study
	%This is our approximating problem. C^\alpha(\overline{\Omega})$ for some $\alpha\in (0,1)
%	Basically, we have transformed a fourth-order equation into a system of two second-order equations, and furthermore, we have discretized the time derivative, thereby turning a parabolic problem into an elliptic one.
	%It is important to note
%	The tricky part here is that we have kept the exponential nonlinearity in (\ref{ot1}). If we had taken the exponential term to be our new unknown function, which seemed to be a very natural choice at first glance, we
%	would have had to show that the new unknown was positive, which was essentially equivalent to trying to find a non-negative solution to the equation $-\Delta\varphi+\tau\varphi=f$ when we had no information
%	on how $f$ changed its sign. Obviously, there is no hope of accomplishing that. It is also interesting that we just need to add the two low-order terms $\tau\psi,\tau u$ in order to ``regularize''	the problem.
	%For each $f\in L^2(\Omega)$, we consider the boundary value problem
	\begin{eqnarray}
		-\des+\rho+\tau\ln\rho &=&f\ \  \mbox{in $\Omega$},\label{ae1}\\
		\nabla\rho\cdot\nu&=&0\ \ \mbox{on $\po$,}\label{ae2}
	\end{eqnarray}
	where $f$ is a given function in $L^2(\Omega)$.
	% Then the function $\ln(\eta_L(s)+L+\tau_0)$ is well-defined. Unfortunately, this function is neither strictly monotone nor coercive. It turns out that it is enough for us to use the following modification	s+L+\ln\tau_0 -\ln(-s-L+\tau_0)+2\ln\tau_0\tau_0\in (0,1),\^{\frac{1}{2}}-L^{\frac{3}{4}}|s|^{\frac{1}{4}}\cap L^\infty(\Omega)
	A weak solution to this problem is a function $\rho\in W^{1,2}(\Omega)$ such that
	\begin{eqnarray}
		\ln\rho &\in& L^{2}(\Omega)\ \ \mbox{and}\label{tta1}\\
		\io \nabla\rho\nabla\varphi dx+\io\rho \varphi+\tau\io\ln\rho \varphi dx&=&\io f\varphi dx\ \ \mbox{for each $\varphi\in  W^{1,2}(\Omega).$}\nonumber
	\end{eqnarray}
	Of course, \eqref{tta1} implies
	$$\rho>0 \ \ \mbox{a.e. on $\Omega$.}$$
	\begin{lem}\label{uex}
		For each $f\in L^2(\Omega)$ there is a unique weak solution to \eqref{ae1}-\eqref{ae2}.
	\end{lem}
\begin{proof} For the existence part, we consider the approximate problem
	\begin{eqnarray}
		-\Delta \rho_\delta +\rho_\delta+\tau\psi_{\delta}(\rho_\delta)&=&f\ \  \mbox{in $\Omega$},\label{ae3}\\
		\nabla\rho_\delta\cdot\nu&=&0\ \ \mbox{on $\po$,}\label{ae4}
	\end{eqnarray}
	where $\delta\in (0,1)$ and
	\begin{equation}
		\psi_{\delta}(s)=
		%,\ \ s\in (-L,\infty).1+\frac{s}{L}+\frac{\delta}{L}
		\left\{\begin{array}{ll}
			\ln\left(s+\delta\right)	&\mbox{if $s> 0$,}\\
			%	\ln(1+\frac{s}{L})&\mbox{if $s>0$,}\\
			\ln\delta	&\mbox{if $s\leq 0$.}
		\end{array}\right.\nonumber
	\end{equation} 
	Existence of a weak solution to this problem is standard, we will omit its proof.
		Next, we proceed to show that we can take $\delta\rightarrow 0$ in \eqref{ae3}-\eqref{ae4}.
	To this end, let 
	$$s_z\equiv1-\delta\in(0,1).$$
	Then we have $$\psi_{\delta}(s_z)=0.$$ Subtract $s_z$ from both sides of \eqref{ae3} and use $\rho_\delta-s_z$ as a test function in the resulting equation to get
	\begin{eqnarray}
		\lefteqn{\io|\nabla\rho_\delta|^2dx+\io(\rho_\delta-s_z)^2dx+\tau\io\psi_{ \delta}(\rho_\delta)(\rho_\delta-s_z)dx}\nonumber\\
		&=&\io (f- s_z)(\rho_\delta-s_z) dx\nonumber\\
	%	&\leq &\|f- s_z\|_{\frac{2N}{N+2}}\|(\rho_\delta-s_z)\|_{\frac{2N}{N-2}}\nonumber\\ \|\nabla\rho_\delta\|_2^2+
	%	&\leq &c\|f- s_z\|_{\frac{2N}{N+2}}\left(\|\nabla\rho_\delta\|_2+\|(\rho_\delta-s_z)\|_2\right)\nonumber\\
		&\leq&\frac{1}{2}\io(\rho_\delta-s_z)^2dx+\frac{1}{2}\io(f- s_z)^2dx.\nonumber
	\end{eqnarray}
	Thus,
	\begin{eqnarray}
	\io|\nabla\rho_\delta|^2dx+\io(\rho_\delta-s_z)^2dx+\tau	\io\psi_{ \delta}(\rho_\delta)(\rho_\delta-s_z) dx\leq c.\label{sa4}
	\end{eqnarray}
Here and in what follows the letter $c$ denotes a positive number independent of $\delta$. 
Note that
$$\psi_{ \delta}(\rho_\delta)(\rho_\delta-s_z)\geq 0\ \ \mbox{a.e. on $\Omega$.}$$
This together with \eqref{sa4} implies that $\{\rho_\delta\}$ is bounded in $W^{1,2}(\Omega)$. We may assume that
	\begin{equation}
		\rho_\delta\rightarrow \rho\ \ \mbox{weakly in $W^{1,2}(\Omega)$, strongly in $L^2(\Omega)$, and a.e. on $\Omega$.}\nonumber
	\end{equation}
Since $\psi_{ \delta}$ is a Lipschitz function, we can use $\psi_{ \delta}(\rho_\delta)$ as a test function in \eqref{ae3} to deduce
	%By suitably modifying the test function  in \eqref{c211} (i.e., use $(|\psi_{ \delta}(\rho_\delta)|+\varepsilon)^{\lambda-1}h_\varepsilon(\rho_\delta-s_z)$ instead),  we can derive that
	%$$\|\tau(\psi_{ \delta}(\rho_\delta)-\psi_{ \delta}(0))\|_2\leq \|f-\tau\psi_{ \delta}(0)\|_2.$$
	%It immediately follows that
	\begin{eqnarray}
	\io\psi_{ \delta}^\prime(\rho_\delta)\left|\nabla\rho_\delta\right|^2dx+\io\rho_\delta\psi_{ \delta}(\rho_\delta)+\tau\io\psi_{ \delta}^2(\rho_\delta)dx=\io f\psi_{ \delta}(\rho_\delta)dx.\label{sa10}
		%&\leq& \|\tau(\psi_{ \delta}(\rho_\delta)-\psi_{ \delta}(0))\|_2+\tau|\psi_{ \delta}(0)||\Omega|^{\frac{1}{2}}\nonumber\\
	%	&\leq &\|f\|_2.
		%+c\tau|\ln(L+\delta)|\leq c.
	\end{eqnarray}
Remember that $\psi_{ \delta}^\prime(\rho_\delta)\geq 0$. Thus, we can conclude from \eqref{sa10} that
\begin{equation}\label{sa11}
	\io\psi_{ \delta}^2(\rho_\delta)dx\leq c(\tau).
\end{equation}
Obviously,
$$ \psi_{ \delta}(\rho_\delta)\rightarrow \left\{\begin{array}{l}
	-\infty\ \ \mbox{a.e. on the set $\{\rho\leq 0\}$,}\\
	\ln \rho\ \ \mbox{a.e. on the set $\{\rho> 0\}$.}
\end{array}\right.$$
%the sequence $\{\psi_{ \delta}(\rho_\delta)\}$ converges a.e. on $\Omega$.	
In view of Fatou's lemma and \eqref{sa11}, we must have 
$$\left|\{\rho\leq 0\}\right|=0$$
and%we have
	$$\io\ln^2\rho dx=\int_{\{\rho>0\}}\ln^2\rho dx\leq \lim_{\delta\rightarrow 0}\io\psi_{ \delta}^2(\rho_\delta)dx\leq c.$$ We are ready to pass to the limit in \eqref{ae3}.
	
	The uniqueness of a weak solution to \eqref{ae1}-\eqref{ae2} is trivial because $\rho+\tau\ln\rho$ is strictly increasing. The proof is complete.
\end{proof}

	\begin{lem}\label{l29}
		%Let (\ref{ot3}) hold, and assume that
		Let $\Omega$ be a bounded domain in $\mathbb{R}^N$ with Lipschitz boundary, and assume that (\ref{ot3}) hold. Then there is a weak solution to (\ref{ot1})-(\ref{ot2}). If, in addition, $v\in L^\infty(\Omega)$, then we have
		\begin{equation}\label{rinfty}
			\ln \rho \in  L^\infty(\Omega).
		\end{equation}
	\end{lem}
	\begin{proof}	We essentially follows the argument in Section 4, \cite{X1}. 
		%A solution will be established via the Leray-Schauder Theorem. To do this, 
		To proceed, we define an operator $B$ from $W^{1,2}(\Omega)$
		into itself as follows: For each $ w\in W^{1,2}(\Omega)$ we first solve the problem
		\begin{eqnarray}
			-\Delta\rho+\rho+\tau\ln\rho  &=&-\frac{w-v}{\tau}+1\ \ \mbox{in $\Omega$},\label{om3}\\
			\nabla\rho\cdot\nu&=&0\ \ \ \mbox{on $\partial\Omega$}.\label{om4}
		\end{eqnarray}
		By Lemma \ref{uex}, there is a unique weak solution $\rho\in W^{1,2}(\Omega)$ with $\ln\rho \in L^2(\Omega)$ to the above problem.  We use the function $\rho$ so obtained to form the problem 
		\begin{eqnarray}
			-\Delta u +\tau u &=&\ln\rho \ \ \ \mbox{in $\Omega$},\label{om5}\\
			\nabla u\cdot\nu&=& 0\ \ \ \mbox{on $\partial\Omega$}.\label{om6}
		\end{eqnarray}
		The classical existence theory asserts that there is a unique weak solution $u\in W^{1,2}(\Omega)$ to \eqref{om5}-\eqref{om6}.	
	We define
		$$%\begin{equation}
		B( w)=u.%\ln(\rho-m_\rho+\tau_0). %\rho-1-\sum_{n=2}^{\infty}\frac{\psi^n}{n!}.
		$$%\end{equation}
		Clearly, $B$ is well-defined.
As in Section 4, \cite{X1}, we can conclude that $B$ is continuous and maps bounded sets into precompact ones. Next, we show that there is a positive number $c$ such that
\begin{equation}
\|u\|_{W^{1,2}(\Omega)}\leq c\label{ot8}
\end{equation}
for all $u\in W^{1,2}(\Omega)$ and $\sigma\in [0,1]$ satisfying
$$u=\sigma B(u).$$
This equation is equivalent to the boundary value problem
\begin{eqnarray}
	-\Delta\rho+\rho+\tau\ln\rho  &=&-\frac{u-v}{\tau}+1 \ \ \ \mbox{in $\Omega$},\label{ot9}\\
-\Delta u+\tau u &=&\sigma\ln\rho \ \ \ \mbox{in $\Omega$},\label{ot10}\\
\nabla u\cdot\nu=\nabla\es\cdot\nu&=& 0\ \ \ \mbox{on $\partial\Omega$}.
\end{eqnarray}
%Use $\rho$ as a test function 
In view of \eqref{rinfty},  we may assume that $\rho$ is bounded away from $0$ below. Thus we can use $\ln\rho$ as a test function in \eqref{ot9} to get
\begin{equation}\label{happy5}
\io\frac{|\nabla\rho|^2}{\rho}dx+	\io(\rho-1)\ln\rho dx+\tau\io\ln^2\rho dx\leq -\frac{1}{\tau}\io(u-v)\ln\rho dx.
\end{equation}
Use $u$ as a test function in \eqref{ot10} to deduce
$$\sigma\io u\ln\rho dx=\io|\nabla u|^2dx+\tau\io u^2dx\geq 0.$$
Drop the first term in \eqref{happy5} and then use the above equation to get
It immediately follows that
\begin{eqnarray}
\io(\rho-1)\ln\rho dx+\tau	\io\ln^2\rho dx\leq\frac{1}{\tau}\io v\ln\rho dx.\nonumber
\end{eqnarray}
Subsequently,
\begin{equation}\label{happy7}
	\io(\rho-1)\ln\rho dx+	\io\ln^2\rho dx\leq c(\tau)\io v^2dx.
\end{equation}
%$$\|\ln\rho\|_{2,\Omega}\leq \frac{1}{\tau^2}\|v\|_{2,\Omega}.$$\leq \frac{1}{\tau}\|v\|_{2,\Omega}\|\ln\rho\|_{2,\Omega}
This together with \eqref{ot10} implies \eqref{ot8}. 
% Here we only mention that

To see \eqref{rinfty}, we first establish
the estimate
\begin{equation}\label{infty2}
	\tau	\|\ln\rho\|_{p,\Omega}\leq 	\left\|\frac{u-v}{\tau}\right\|_{p,\Omega} \ \ \mbox{for each $p\geq 2$.}
\end{equation}
 For this purpose, we introduce the function
\begin{equation}
	h_\varepsilon(s)=\left\{\begin{array}{ll}
		1&\mbox{if $s>\varepsilon$,}\\
		s&\mbox{if $|s|\leq \varepsilon$,}\\
		-1&\mbox{if $s<-\varepsilon$,}\ \ \ \varepsilon>0.
	\end{array}\right.\nonumber
\end{equation}
Use $|\ln\rho|^{p-1}h_\varepsilon(\rho-1)$ as a test function in \eqref{om3} to derive
$$\io|\ln\rho|^{p-1}h_\varepsilon(\rho-1)(\rho-1)dx+\tau\io\ln\rho|\ln\rho|^{p-1} h_\varepsilon(\rho-1)dx\leq -\io\frac{u-v}{\tau}|\ln\rho|^{p-1}h_\varepsilon(\rho-1)dx.$$
Here we have used the fact that $|\ln\rho|^{p-1}h_\varepsilon(\rho-1)$ is an increasing function of $\rho$. Taking $\varepsilon\rightarrow 0$ yields
$$\tau\io|\ln\rho|^{p}\leq \io\left|\frac{u-v}{\tau}\right|\frac{u-v}{\tau}|\ln\rho|^{p-1}h_\varepsilon(\rho-1)dx\leq\left\|\frac{u-v}{\tau}\right\|_{p,\Omega}	\|\ln\rho\|_{p,\Omega}^{p-1}.$$
The estimate \eqref{infty2} follows. Now take $p\rightarrow \infty$ in \eqref{infty2} to get
\begin{equation}\label{infty3}
	\tau	\|\ln\rho\|_{\infty,\Omega}\leq 	\left\|\frac{u-v}{\tau}\right\|_{\infty,\Omega}.
\end{equation}
Lemma \ref{wbd} asserts that for each $q>\frac{N}{2}$ there is a positive number $c=c(N, \Omega, \tau)$ such that
$$	\|u\|_{\infty,\Omega}\leq c\|u\|_{1,\Omega}+c	\|\ln\rho\|_{q,\Omega}\leq c	\|\ln\rho\|_{q,\Omega}.$$
This combined with \eqref{infty3} implies
\begin{eqnarray}
	\|\ln\rho\|_{\infty,\Omega}&\leq &c\|u\|_{\infty,\Omega} +c\|v\|_{\infty,\Omega} \nonumber\\
	&\leq &c	\|\ln\rho\|_{q,\Omega} +c\|v\|_{\infty,\Omega} \nonumber\\
	&\leq &\varepsilon	\|\ln\rho\|_{\infty,\Omega}+c(\varepsilon)\|\ln\rho\|_{1,\Omega}+c\|v\|_{\infty,\Omega},\ \ \varepsilon>0.\nonumber
\end{eqnarray}
The last step is due to the interpolation inequality (1) in Lemma \ref{l23}. Taking $\varepsilon$ suitably small yields \eqref{rinfty}.
The proof is complete.
\end{proof}	
Note that for Lemma \ref{l29} we do not have to assume (H2).

To conclude this section, we would like to make some remarks about the possible non-negativity of $u$. Since $u$ represents the surface height in our model, it is natural for us to expect
% From the modeling perspective, one would expect to have
$$u\geq 0.$$
%since it represents the surface height. 
In this regard, one is tempted to consider the following approximation
	\begin{eqnarray}
	-\des+\rho+\tau\ln\rho  &=&-\frac{u-v}{\tau}+1 \ \ \ \mbox{in $\Omega$,}\label{aot1}\\
	-\Delta u+\tau \ln u &=&\ln\rho  \ \ \ \mbox{in $\Omega$},\label{aot4}\\
	\nabla u\cdot\nu=\nabla\es\cdot\nu&=& 0\ \ \ \mbox{on $\partial\Omega$}.\label{aot2}
\end{eqnarray}
%In fact, we can show that 
It turns out that this problem does have a solution. The proof is only a slight modification of that for Lemma \ref{l29}. To see this, we define an operator $B$ from $W^{1,2}(\Omega)$
into itself as follows: For each $ w\in W^{1,2}(\Omega)$ we first solve the problem
\begin{eqnarray}
	-\Delta\rho+\rho+\tau\ln\rho  &=&-\frac{w-v}{\tau}+1\ \ \mbox{in $\Omega$},\nonumber\\
	\nabla\rho\cdot\nu&=&0\ \ \ \mbox{on $\partial\Omega$}.\nonumber
\end{eqnarray}
%By Lemma \ref{uex}, there is a unique weak solution $\rho\in W^{1,2}(\Omega)$ with $\ln\rho \in L^2(\Omega)$ to the above problem.  
We use the function $\rho$ so obtained to form the problem 
\begin{eqnarray}
	-\Delta u +\tau\ln u &=&\ln\rho \ \ \ \mbox{in $\Omega$},\nonumber\\
	\nabla u\cdot\nu&=& 0\ \ \ \mbox{on $\partial\Omega$}.\nonumber
\end{eqnarray}
%The classical existence theory asserts that there is 
By replacing the second term in \eqref{ae1} by $\delta\rho$  and then taking $\delta\rightarrow 0$ in the resulting problem, we can also conclude that there is a unique weak solution to the preceding problem. See \cite{X1} for details.
%By our earlier remark, this problem also has a unique weak solution $u$ in the space $W^{1,2}(\Omega)$.
% to \eqref{om5}-\eqref{om6}.	
Define
$$%\begin{equation}
B( w)=u.%\ln(\rho-m_\rho+\tau_0). %\rho-1-\sum_{n=2}^{\infty}\frac{\psi^n}{n!}.
$$%\end{equation}
Clearly, $B$ is well-defined.
Once again, we can infer from Section 4 in \cite{X1} that  $B$ is continuous and maps bounded sets into precompact ones.

Next, we show that there is a positive number $c$ such that
\begin{equation}
	\|u\|_{W^{1,2}(\Omega)}\leq c\label{aot8}
\end{equation}
for all $\psi\in W^{1,2}(\Omega)$ and $\sigma\in [0,1]$ satisfying
$$u=\sigma B(u).$$
This equation is equivalent to the boundary value problem
\begin{eqnarray}
	-\Delta\rho+\rho+\tau\ln\rho  &=&-\frac{u-v}{\tau}+1 \ \ \ \mbox{in $\Omega$},\label{aot9}\\
	-\Delta u+\tau\sigma(\ln u-\ln\sigma) &=&\sigma\ln\rho \ \ \ \mbox{in $\Omega$},\label{aot10}\\
	\nabla u\cdot\nu=\nabla\es\cdot\nu&=& 0\ \ \ \mbox{on $\partial\Omega$}.
\end{eqnarray}
Integrate \eqref{aot9} over $\Omega$ to get
$$\io udx=-\tau\io\rho dx-\tau^2\io\ln\rho dx+\io vdx+\tau|\Omega|.$$
Obviously, \eqref{happy7} is still valid. Use it in the above equation to get
%This together with \eqref{june1} implies
$$\left|\io udx\right|\leq c.$$
Use $u-1$ as a test function in \eqref{aot10} to get
$$\io|\nabla u|^2dx\leq \io\left(\sigma\ln\rho+\tau\sigma\ln\sigma\right)u dx\leq \varepsilon\io u^2dx+c(\varepsilon).$$
We deduce from Poincar\'{e}'s inequality that
\begin{eqnarray}
	\io u^2dx&\leq &2\io\left(u-\frac{1}{|\Omega|}\io udx\right)^2dx+\frac{2}{|\Omega|}\left(\io udx\right)^2\nonumber\\
	&\leq &c\io|\nabla u|^2dx+c\leq c\varepsilon	\io u^2dx+c.\nonumber
\end{eqnarray}
By taking $\varepsilon$ suitably small, we obtain
%This together with \eqref{ot10} implies 
\eqref{aot8}.

Unfortunately, when we try to pass to the limit in the system \eqref{aot1}-\eqref{aot2}, we run into an insurmountable problem. That is, the existence of a non-negative solution to \eqref{p1} remains open. This is probably not a surprise because it is well known that the bi-harmonic heat equation does not satisfy the maximum principle, i.e., solutions change sign no matter how one prescribes the initial boundary conditions. However,  certain nonlinearities in fourth-order equations can allow the existence of non-negative solutions \cite{LX2,X2}.
	\section{Proof of Theorem \ref{thm1}}
	\setcounter{equation}{0}
	
	The proof of Theorem \ref{thm1} will be divided into several lemmas. First, we present our approximation scheme. This is based upon Lemma \ref{l29}. Then
	we proceed to derive estimates similar to those in Subsection \ref{sub2} for our approximate problems. These estimates are shown to be sufficient to justify passing to the limit.

	Let $T>0$ be given. For each $j\in\{1,2,\cdots,\}$ we divide the time interval $[0,T]$ into $j$ equal sub-intervals. Set
	$$\tau=\frac{T}{j}.$$ We assume that $j$ is so large that
	\begin{equation}\label{tcon}
		%	\tau<1.
	\tau<\min\left\{1,\frac{1}{\|u_0\|_{W^{2,2}(\Omega)}},\frac{1}{8T}\right\}.
	\end{equation}
	%We discretize  (\ref{p1})-(\ref{p3}) as follows. 
	Let $u_0$ be given as in (H3). For $k=1,\cdots, j$, we solve recursively the system
	\begin{eqnarray}
	\frac{u_k-u_{k-1}}{\tau}+\esk-\Delta\esk+\tau\sk&=&1\ \ \ \mbox{in $\Omega$},\label{s31}\\
	-\Delta\uk+\tau\uk&=&\sk\ \ \ \mbox{in $\Omega$},\label{s32}\\
	\nabla\esk\cdot\nu=\nabla\uk&=&0\ \ \ \mbox{on $\partial\Omega$}.\label{s33}
	\end{eqnarray}
	Set $$t_k=k\tau.$$  
	We can form the following functions on $\Omega_T$ by setting
	\begin{eqnarray}
	\utj(x,t)&=&\frac{t-t_{k-1}}{\tau}\uk(x)+\left(1-\frac{t-t_{k-1}}{\tau}\right)\uko(x), \ \\
	\ubj(x,t)&=&\uk(x),\\
	\rbj(x,t)&=&\eskx,\\
	\rtj(x,t)&=&\frac{t-t_{k-1}}{\tau}\esk(x)+\left(1-\frac{t-t_{k-1}}{\tau}\right)\esko(x),  \\
	\gbj(x,t) &=&\frac{u_k-u_{k-1}}{\tau}+\esk-1\equiv G_k,\label{gkd}\\
	\ctj(x,t)&=&\frac{t-t_{k-1}}{\tau}\sqrt{\eskx}+\left(1-\frac{t-t_{k-1}}{\tau}\right)\sqrt{\esko(x)}
	\end{eqnarray}
whenever  $x\in\Omega,  \ t\in(t_{k-1},t_k]$.
%where we take
In the last equation, we take
\begin{equation}\label{om11}
	\rho_0=e^{-\Delta u_0+\tau u_0}.
\end{equation}
%This together with \eqref{bcon} shows that $u_0$ satisfies
%\begin{eqnarray}
%	-\Delta u_0+\tau u_0&=&\ln 	\rho_0\ \ \mbox{in $\Omega$,}\\
%	\nabla u_0\cdot\nu&=& 0\ \ \mbox{on $\po$.}
%\end{eqnarray}
	Subsequently, we can rewrite (\ref{s31})-(\ref{s33}) as
	\begin{eqnarray}
	\partial_t\utj-\Delta \rbj+\rbj+\tau\ln\rbj&=&1\ \ \ \mbox{in $\omt$},\label{omm1}\\
	-\Delta\ubj+\tau\ubj &=&\ln\rbj \ \ \ \mbox{in $\omt$}.\label{omm22}
	\end{eqnarray}
	We proceed to derive a priori estimates for the sequences $\{\utj,\ubj,\rbj,\rtj,\ctj,\gbj,\ln\rbj\}$. The discretized version of (\ref{nm1}) is the following
	\begin{lem}\label{l31}
	We have
	\begin{eqnarray}
			\lefteqn{\int_{\Omega_T}\left(\left(\Delta\rbj\right)^2 +(\gbj+\tau\ln\rbj)^2+2\left|\nabla\esbj\right|^2+8\tau\left|\nabla\sqrt{\esbj}\right|^2\right)\, dxdt+2\sup_{0\leq t\leq T}\int_\Omega(\rbj-\ln\rbj) dx}\nonumber\\
		&&
		+2\tau\iot|\nabla\esbj|^2 \, dxdt+2\tau\iot (\esbj-1)^2 dxdt+2\tau^2 \iot (\esbj-1) \ln\rbj \, dxdt\nonumber\\
		%	&&+ 2\tau\iot(\esbj-1)\ln \esbj dxdt\nonumber\\
	%	&\leq & 4\int_\Omega \left(e^{-\Delta u_0(x)+\tau u_0(x)}+\Delta u_0-\tau u_0\right) \, dx\leq 4e^{\tau\|u_0\|_{\infty,\Omega}}\io e^{-\Delta u_0(x)}dx-4\tau\io u_0 dx
		%	-4\tau\int_{\{\rbj\leq1\}} (\esbj-1)\esbj dxdt-4\tau^2 \int_{\{\rbj\leq1\}}\,  \esbj \ln\rbj \, dxdt
		\nonumber\\
		&\leq &c(\Omega, N)\|e^{-\Delta u_0(x)}\|_{1,\Omega}+4\tau\|u_0\|_{1,\Omega}.
		\end{eqnarray}
	Here $c$ depends only on $N, \Omega$.
\end{lem}
\begin{proof} Using \eqref{gkd}, we can write
	\eqref{s31} as
	\begin{equation}\label{ss31}
		G_k+\tau\ln\esk-\Delta\esk=0\ \ \mbox{in $\Omega$.}
	\end{equation}
	Square both sides of this equation and integrate the resulting equation  over $\Omega$ to derive
	\begin{eqnarray}
		\io\left[(G_k+\tau\ln\esk)^2+\left(\Delta \esk\right)^2\right]dx
		-2\io (G_k+\tau\ln\esk)\Delta \esk dx=0.\label{ota4}
		%\nonumber\\	&=&with respect to $x$
	\end{eqnarray} 
	We easily see that
	\begin{equation}\label{ota7}
		-2\io (G_k+\tau\ln\esk)\Delta \esk dx=-2\io\frac{u_k-u_{k-1}}{\tau}\Delta \esk dx+2\io\left|\nabla\esk\right|^2dx+8\tau\io|\nabla\sqrt{\esk}|^2dx.
	\end{equation}
	Thus we only need to be concerned with 
	%To evaluate 
	the second integral in the above equation.
	%, %we assume that $t=t_k$ for some $k\in \{1,2,\cdots,j\}$.
	%Multiply through \eqref{omm1} by $\esk+L$ and integrate the resulting equation over $\Omega$ to obtain
	%\begin{equation}
	%	\io\partial_t\utj(\rbj+L)dx=-\io\left|\nabla\rbj\right|^2dx-\io(\rbj+L)^2dx-\tau\io(\rbj+L)\ln(\rbj+L)dx+\io(\rbj+L)dx.
	%	\end{equation}
	%	Assume $t\in(t_{k-1},t_k]$ for some $k\in \{1,2,\cdots,j\}$.
	For this purpose,	we  use $\tau  (\esk-1) $ as a test function in (\ref{s31}) to
	yield
	\begin{eqnarray}
		\int_\Omega   (\esk-1) (\uk-\uko) \, dx&=&-\tau\int_\Omega|\nabla\esk|^2 \, dx-\tau\io (\esk-1)^2 dx\nonumber\\	&&-\tau^2\int_\Omega   (\esk-1) \sk \, dx.\label{ota2}
	\end{eqnarray}
	%	We easily obtain from (\ref{s32}) that
	On the other hand, we can conclude from \eqref{s32} and \eqref{s33} that
	\begin{eqnarray}
		-\Delta(\uk-\uko)+\tau(\uk-\uko)&=&\sk-\sko\ \ \mbox{in $\Omega$},\label{ota3}\\
		\nabla(\uk-\uko)\cdot\nu=0\ \ \mbox{on $\po$ }
	\end{eqnarray}
	Note that the above system also holds for $k=1$ due to \eqref{om11} and \eqref{bcon}. 
	%implies$$\nabla(\uk-\uko)\cdot\nu=0\ \ \mbox{on $\po$ for each $k\in \{1,2,\cdots, j\}$.}$$
	%	Keeping (\ref{ota3}), (\ref{ota1}), and (\ref{ota2}) 
	With these in mind, we estimate
	\begin{eqnarray}
		%-2\io \pt\utj\Delta\rbj&=&-\frac{2}{\tau}\int_\Omega\Delta  \esk(\uk-\uko)\, dx\nonumber\\
		-2\io\frac{u_k-u_{k-1}}{\tau}\Delta \esk dx &=&
		-\frac{2}{\tau}\int_\Omega  (\esk-1)\Delta(\uk-\uko)\, dx\nonumber\\
		&=&\frac{2}{\tau}\int_\Omega  (\esk-1)(\sk-\sko)\, dx-2\int_\Omega  (\esk-1)(\uk-\uko)\, dx\nonumber\\
		&\geq&\frac{2}{\tau}\int_\Omega\left[(\esk-\ln\esk)-(\esko-\ln\esko)\right]\, dx+2\tau\int_\Omega|\nabla\esk|^2 \, dx\nonumber\\
		&&+2\tau\io (\esk-1)^2 dx
		+2\tau^2\int_\Omega   (\esk-1) \sk\,  dx.\label{ota6}
	\end{eqnarray}
	The last step is due to \eqref{ota1}.
	Collecting (\ref{ota7}) and (\ref{ota6}) in (\ref{ota4}) gives
	\begin{eqnarray}
		\lefteqn{\int_\Omega\left(\left(\Delta\esk\right)^2 +(G_k+\tau\ln\esk)^2+2\left|\nabla\esk\right|^2+8\tau\left|\nabla\sqrt{\esk}\right|^2\right)\, dx}\nonumber\\
		&&+\frac{2}{\tau}\int_\Omega\left[(\esk-\ln\esk)-(\esko-\ln\esko)\right] dx
		+2\tau\int_\Omega|\nabla\esk|^2 \, dx\nonumber\\
		&&+2\tau\io (\esk-1)^2 dx+2\tau^2\int_\Omega \,    (\esk-1) \sk \, dx
		\leq 0.\nonumber
	\end{eqnarray}
	Multiplying through the inequality by $\tau$ and summing up the resulting one over $k$, we obtain
	% that for each $s\in\{\tau, 2\tau,\dots, j\tau\}$
	%there holds
	\begin{eqnarray}
		\lefteqn{\int_{\Omega_T}\left(\left(\Delta\rbj\right)^2 +(\gbj+\tau\ln\rbj)^2+2\left|\nabla\esbj\right|^2+8\tau\left|\nabla\sqrt{\esbj}\right|^2\right)\, dxdt+2\sup_{0\leq t\leq T}\int_\Omega(\rbj-\ln\rbj) dx}\nonumber\\
		&&
		+2\tau\iot|\nabla\esbj|^2 \, dxdt+2\tau\iot (\esbj-1)^2 dxdt+2\tau^2 \iot (\esbj-1) \ln\rbj \, dxdt\nonumber\\
		%	&&+ 2\tau\iot(\esbj-1)\ln \esbj dxdt\nonumber\\
		&\leq & 4\int_\Omega \left(e^{-\Delta u_0(x)+\tau u_0(x)}+\Delta u_0-\tau u_0\right) \, dx\leq 4e^{\tau\|u_0\|_{\infty,\Omega}}\io e^{-\Delta u_0(x)}dx-4\tau\io u_0 dx
	%	-4\tau\int_{\{\rbj\leq1\}} (\esbj-1)\esbj dxdt-4\tau^2 \int_{\{\rbj\leq1\}}\,  \esbj \ln\rbj \, dxdt
	\nonumber\\
		&\leq &c(\Omega, N)\|e^{-\Delta u_0(x)}\|_{1,\Omega}+4\tau\|u_0\|_{1,\Omega}.\nonumber
	\end{eqnarray}
	The last step is due to \eqref{tcon} and \eqref{emb}.
	This finishes the proof.
	% It is not difficult to see that the above inequality is valid for each $s\in (0,T]$. Also, the function $\theta(s)=s e^s\geq -e^{-1}$ for $s\in(-\infty, 0)$, and thus the desired result follows.
\end{proof}

An immediate consequence of this lemma is
\begin{equation}\label{happy2}
	\sup_{0\leq t\leq T}\int_\Omega(\rbj+|\ln\rbj|) dx\leq c(\Omega, N)\varepsilon_0+\left(16\|u_0\|_{1,\Omega}+4|\Omega|T\right)\tau.
\end{equation}
To see this, we first integrate (\ref{omm22}) over $\Omega$ to obtain
$$\tau\int_{\Omega}\ubj \, dx=\int_{\Omega}\ln\rbj \, dx.$$ 
Then we calculate
\begin{eqnarray}
	\io\rbj dx&=&\io\left(\rbj-\ln\rbj\right)dx+\io\ln\rbj dx\nonumber\\
	&\leq&c(\Omega, N)\varepsilon_0+2\|u_0\|_{1,\Omega}\tau+\tau\io\ubj dx.\label{happy1}
\end{eqnarray}
On the other hand,
\begin{eqnarray}
	\int_{\Omega}|\ln\rbj|\, dx &=&
	\int_{\Omega}\left(\ln^+\rbj +\ln^-\rbj \right) dx\nonumber\\
	&=&-2\int_{\{\rbj<1\}}\ln\rbj  \, dx
	+\int_{\Omega}\ln\rbj  \, dx\nonumber\\
	&\leq &2\int_{\Omega}(\rbj-\ln\rbj)\, dx+\tau\int_{\Omega}\ubj \, dx\nonumber\\
	&\leq&c(\Omega, N)\varepsilon_0+4\|u_0\|_{1,\Omega}\tau+\tau\int_{\Omega}\ubj \, dx .\nonumber
\end{eqnarray}
Here we have used the fact that $\rbj-\ln\rbj>0$. Adding this inequality to \eqref{happy1} gives 
\begin{equation}\label{happy3}
	\io\rbj dx+	\int_{\Omega}|\ln\rbj|\, dx\leq c(\Omega, N)\varepsilon_0+6\|u_0\|_{1,\Omega}\tau+2\tau\int_{\Omega}\ubj \, dx. 
\end{equation}
We integrate \eqref{s31} over $\Omega$ to derive
$$\io\frac{u_k-u_{k-1}}{\tau}dx=-\io\esk dx-\tau\io\ln\esk dx+|\Omega|.$$
Multiply through this equation by $\tau$ and sum up the resulting equation over $k$ to get
\begin{eqnarray}
	\sup_{0\leq t\leq T}\left|\io\ubj dx\right|&\leq& \left|\io u_0 dx\right|+\iot\rbj dxdt+\tau\iot\left|\ln\rbj\right| dxdt+|\Omega|T\nonumber\\
	&\leq &\|u_0\|_{1,\Omega}+T\sup_{0\leq t\leq T}\io\rbj dx+\tau T\sup_{0\leq t\leq T}\io|\ln\rbj|dx+|\Omega|T.\label{ubb}
\end{eqnarray}
Keeping this in mind, we deduce from \eqref{happy3} that
\begin{eqnarray}
	\sup_{0\leq t\leq T}\int_\Omega(\rbj+|\ln\rbj|) dx&\leq &c(\Omega, N)\varepsilon_0+6\|u_0\|_{1,\Omega}\tau+2\tau\sup_{0\leq t\leq T}\left|\int_{\Omega}\ubj \, dx\right|\nonumber\\
	&\leq & c(\Omega, N)\varepsilon_0+8\|u_0\|_{1,\Omega}\tau+4\tau T\sup_{0\leq t\leq T}\int_\Omega(\rbj+|\ln\rbj|) dx\nonumber\\
	&&+2|\Omega|T\tau.
\end{eqnarray}
According to \eqref{tcon},  $4\tau T<\frac{1}{2}$. Use this in the above inequality to get \eqref{happy2}.

				Now we are ready to obtain a discretized version of (\ref{nm3}).
				\begin{lem}\label{l34}
					We have
					\begin{eqnarray}
						\lefteqn{\sup_{0\leq t\leq T}\int_{\Omega}
							\left(\frac{1}{2}\gbj^2+\left(\frac{1}{2}+\frac{\tau}{2}\right)|\nabla\esbj|^2+\frac{\tau}{2}(\esbj-1)^2\right)dx+2\iot\left(\pt\tilde{\sigma_j}\right)^2\, dxdt}\nonumber\\
						%	&&+\frac{1}{2}\int_{\Omega}\left((\esk-1)^2-(\esko-1)^2\right)\, dx\nonumber\\
						&&+\sup_{0\leq t\leq T}\tau^2\int_{\{\rbj>1\}}\esbj\ln\esbj dx+\sup_{0\leq t\leq T}\tau\io\left(\rbj-\ln\esbj\right)\, dx\nonumber\\
						&\leq&c\varepsilon_0+(c+cT)\tau.
					\end{eqnarray}
				Here $c$ depends only on $\Omega, N$.
				\end{lem}
				\begin{proof}Define
					\begin{equation}\label{godf}
						G_0=\Delta \rho_0-\tau\ln \rho_0.
					\end{equation}
				This combined with (\ref{ss31}) implies that
					\begin{equation}\label{eqfg}
					\frac{G_k-G_{k-1}}{\tau}-\Delta\left(\frac{\esk-\rho_{k-1}}{\tau}\right)+\sk-\ln\rho_{k-1}=0\ \ \mbox{in $\Omega$\ \ \mbox{for each $k\in\{1, 2,3,\cdots, j\}$}.}
					\end{equation}
				We can easily derive from \eqref{bcon} that
				$$\nabla\rho_0\cdot\nu=0\ \ \mbox{on $\po$.}$$
				Thus, we can use $G_k$ as a test function in \eqref{eqfg} (even for $k=1$) to get
					%Multiply through this equation by $\frac{u_k-u_{k-1}}{\tau}$ and integrate the resulting equation over $\Omega$ to obtain
					\begin{eqnarray}
					\lefteqn{\frac{1}{\tau}\int_\Omega
					\left(G_k-G_{k-1}\right)G_k dx}\nonumber\\
				&&+\frac{1}{\tau}\int_\Omega \nabla G_k\cdot\nabla
					\left(\esk-\rho_{k-1}\right) \, dx
					+\int_\Omega \left(\sk-\ln\rho_{k-1}\right)G_k\, dx =0.\label{otn4}
					\end{eqnarray}
				Once again, we can use (3) in Lemma \ref{l24} to handle the first term. The second integral in \eqref{otn4} can be evaluated as follows:
					\begin{eqnarray}
				\lefteqn{\frac{1}{\tau}\int_\Omega \nabla G_k\cdot\nabla
				\left(\esk-\rho_{k-1}\right) \, dx}\nonumber\\&=&
				-\frac{1}{\tau}\int_\Omega\left(\esk-\rho_{k-1}\right)
				\Delta\left(\frac{u_k-u_{k-1}}{\tau}\right) \, dx+\frac{1}{\tau}\int_\Omega \nabla \esk\cdot\nabla
				\left(\esk-\rho_{k-1}\right) \, dx\nonumber\\
				&\geq&\frac{1}{\tau^2}\int_\Omega\left(\esk-\rho_{k-1}\right)
				\left(\sk-\ln\rho_{k-1}\right) \, dx-\frac{1}{\tau}\int_\Omega\left(\esk-\rho_{k-1}\right)
				\left(u_k-u_{k-1}\right)\, dx\nonumber\\
				&&+\frac{1}{2\tau}\int_\Omega\left(|\nabla\esk|^2-|\nabla\esko|^2\right)dx
				\nonumber\\
					&\geq&\frac{2}{\tau^2}\int_\Omega\left(\sqrt{\esk}-\sqrt{\esko}\right)^2 dx-\frac{1}{\tau}\int_\Omega\left(\esk-\rho_{k-1}\right)
				\left(u_k-u_{k-1}\right)\, dx\nonumber\\
				&&+\frac{1}{2\tau}\int_\Omega\left(|\nabla\esk|^2-|\nabla\esko|^2\right)dx
				%&\geq& -\frac{1}{\tau}\int_\Omega\left(\esk-\rho_{k-1}\right)
				%\left(u_k-u_{k-1}\right)\, dx+\frac{1}{2\tau}\int_\Omega\left(|\nabla\esk|^2-|\nabla\esko|^2\right)dx
				.\label{otn2}
\end{eqnarray}
The last step is due to (\ref{om2}).
To estimate the second to last integral in (\ref{otn2}), we use $\esk-\esko$ as a test function in (\ref{s31}) and then apply (3) and \eqref{om1} in Lemma \ref{l24} to obtain
\begin{eqnarray}
\lefteqn{-\frac{1}{\tau}\int_{\Omega}\left(u_k-u_{k-1}\right)(\esk-\esko)\, dx}\nonumber\\
&=&\int_{\Omega}\nabla\esk\nabla(\esk-\esko)\, dx+\io(\esk-1)(\esk-\esko)\, dx\nonumber\\
&&+\tau\int_{\Omega}\sk(\esk-\esko)\, dx\nonumber\\
&\geq&\frac{1}{2}\int_{\Omega}\left(|\nabla\esk|^2-|\nabla\esko|^2\right)\, dx+\frac{1}{2}\int_{\Omega}\left((\esk-1)^2-(\esko-1)^2\right)\, dx\nonumber\\
&&+\tau\int_{\Omega}\left(\esk\sk-\esko\sko\right)\, dx
-\tau\int_{\Omega}\left(\esk-\esko\right) \,dx.\nonumber
\end{eqnarray}
 Calculating the third integral in (\ref{otn4}), we invoke \eqref{ota1} to obtain
\begin{eqnarray}
\int_\Omega \left(\sk-\ln\rho_{k-1}\right)G_k\, dx&=&\frac{1}{\tau} \int_{\Omega}(\sk-\sko)(\uk-\uko)\, dx\nonumber\\
&&+\int_{\Omega}(\sk-\sko)(\esk-1)dx\nonumber\\
&\geq&\int_{\Omega}\left(-\Delta(\uk-\uko)+\tau(\uk-\uko)\right)(\uk-\uko)\, dx\nonumber\\
&&+\io(\esk-\esko)dx-\int_{\Omega}(\sk-\sko)dx\nonumber\\
&\geq &\io\left[(\esk-\sk)-(\esko-\sko)\right]dx
%\int_{\Omega}|\nabla(\uk-\uko)|^2\, dx+\tau\int_{\Omega}|\uk-\uko|^2 \, dx
.\nonumber
\end{eqnarray}
Using the preceding results in (\ref{otn4}) yields
\begin{eqnarray}
\lefteqn{\frac{1}{2\tau}\int_{\Omega}
\left(G_k^2-G_{k-1}^2\right)\,dx+\left(\frac{1}{2\tau}+\frac{1}{2}\right)\int_{\Omega}\left(|\nabla\esk|^2-|\nabla\esko|^2\right)\, dx}\nonumber\\
&&+2\int_{\Omega}\left(\frac{\sqrt{\esk}-\sqrt{\esko}}{\tau}\right)^2\, dx+\frac{1}{2}\int_{\Omega}\left((\esk-1)^2-(\esko-1)^2\right)\, dx\nonumber\\
&&+\tau\int_{\Omega}\left(\esk\sk-\esko\sko\right)\, dx-\tau\int_{\Omega}\left(\esk-\esko\right) \, dx\nonumber\\
%&&-\tau\int_{\Omega}\left(\esk-\esko\right) \,dx
%+ 2\int_{\Omega}\left(\frac{e^{\frac{1}{2}\sk}-e^{\frac{1}{2}\sko}}{\tau}\right)^2\, dx\nonumber\\
%&&+\frac{1}{2}\int_{\Omega}\left(|\nabla\esk|^2-|\nabla\esko|^2\right)\, dx
%+\tau\int_{\Omega}\left(\esk\sk-\esko\sko\right)\, dx\nonumber\\ 
&&+\io\left[(\esk-\sk)-(\esko-\sko)\right]dx
\leq 0.\nonumber
\end{eqnarray}
Multiply through the inequality by $\tau$, sum the resulting inequality over $k$, and thereby obtain
\begin{eqnarray}
	\lefteqn{\sup_{0\leq t\leq T}\int_{\Omega}
		\left(\frac{1}{2}\gbj^2+\left(\frac{1}{2}+\frac{\tau}{2}\right)|\nabla\esbj|^2+\frac{\tau}{2}(\esbj-1)^2\right)dx+2\iot\left(\pt\tilde{\sigma_j}\right)^2\, dxdt}\nonumber\\
%	&&+\frac{1}{2}\int_{\Omega}\left((\esk-1)^2-(\esko-1)^2\right)\, dx\nonumber\\
	&&+\sup_{0\leq t\leq T}\tau^2\int_{\{\rbj>1\}}\esbj\ln\esbj dx+\sup_{0\leq t\leq T}\tau\io\left(\rbj-\ln\esbj\right)\, dx\nonumber\\%+(1-\tau)\sup_{0\leq t\leq T}\int_{\Omega}\esbj \, dx\nonumber\\
	%&&-\tau\int_{\Omega}\left(\esk-\esko\right) \,dx\sup_{0\leq t\leq T}\tau^2\int_{\{\rbj>1\}}\esbj\ln\esbj-\, dx
	%+ 2\int_{\Omega}\left(\frac{e^{\frac{1}{2}\sk}-e^{\frac{1}{2}\sko}}{\tau}\right)^2\, dx\nonumber\\
	%&&+\frac{1}{2}\int_{\Omega}\left(|\nabla\esk|^2-|\nabla\esko|^2\right)\, dx
	%+\tau\int_{\Omega}\left(\esk\sk-\esko\sko\right)\, dx\nonumber\\ 
	&\leq &\io G_0^2dx+\io|\nabla\rho_0|^2dx+\tau\io(\rho_0-1)^2dx-2\sup_{0\leq t\leq T}\tau^2\int_{\{\rbj\leq 1\}}\esbj\ln\esbj dx\nonumber\\
	&&+2\tau^2\sup_{0\leq t\leq T}\io\rbj dx+\tau\io(\rho_0-\ln\rho_0) dx.\label{june2}
\end{eqnarray}			
We calculate from \eqref{om11} and \eqref{godf} that
\begin{eqnarray}
	\nabla\rho_0&=&e^{\tau u_0}\nabla e^{-\Delta u_0}+\tau e^{-\Delta u_0}e^{\tau u_0}\nabla u_0,\nonumber\\
	\Delta\rho_0&=&e^{\tau u_0}\Delta e^{-\Delta u_0}+2\tau \nabla e^{-\Delta u_0}\cdot e^{\tau u_0}\nabla u_0+\tau e^{-\Delta u_0+\tau u_o}\left(\tau|\nabla u_0|^2+\Delta u_0\right).\nonumber
\end{eqnarray}
We can derive from \eqref{emb} and \eqref{tcon} that
\begin{eqnarray}
	\io|\nabla\rho_0|^2dx\leq c\varepsilon_0^2,\ \ \io G_0^2dx\leq c\varepsilon_0^2.\nonumber
\end{eqnarray}
Here $c$ depends only on $\Omega, N$. Moreover,
\begin{eqnarray}
	\io\rho_0^2dx&=&\io e^{-2\Delta u_0+2\tau u_0}dx\leq c\varepsilon_0^2,\nonumber\\
	\io(\rho_0-\ln\rho_0) dx&=&\io e^{-\Delta u_0+\tau u_0}dx -\tau\io u_0 dx\leq c\varepsilon_0+1,\nonumber\\
	\tau^2\sup_{0\leq t\leq T}\io\rbj dx&\leq &c\varepsilon_0+c\tau+cT\tau^3.\nonumber
\end{eqnarray}
The last inequality is due to Lemma \ref{l31}. Collecting these estimates in \eqref{june2} gives the lemma.
	\end{proof}
			\begin{lem}\label{l35}
				The sequence $\{\rtj\}$ is bounded $W^{1,2}(\Omega_T)$.
			\end{lem}	
		\begin{proof} Since $u_0\in L^\infty(\Omega)$ we can infer from Lemma \ref{l29} that
			\begin{equation}
				\ln\rbj(\cdot,t)\in L^\infty(\Omega)\ \ \mbox{for each $j$ and each $t\in[0,T]$.}\nonumber
			\end{equation}
		Thus we can use $\tau\ln\esk$ as a test function in \eqref{ss31} to get
		\begin{equation}
			\io \tau^2\ln^2\esk dx\leq -\io G_k\tau\ln\esk dx\leq \frac{1}{2}\io G_k^2dx+\frac{1}{2}\io\tau^2\ln^2\esk dx,\nonumber
		\end{equation}
	from whence follows
	\begin{equation}\label{rg}
		\io \tau^2\ln^2\rbj dx\leq \io \gbj^2 dx.
	\end{equation}
By virtue of (H2) and Lemma \ref{wbd},  there is a positive number $c=c(\Omega)$ such that
\begin{eqnarray}
	\|\rbj(\cdot, t)\|_{\infty,\Omega}&\leq& c\|\rbj(\cdot, t)\|_{1,\Omega}+c\|\gbj+\tau\ln\rbj\|_{2,\Omega}\nonumber\\
	&\leq &c\|\rbj(\cdot, t)\|_{1,\Omega}+c\|\gbj\|_{2,\Omega}\leq c.\label{rg2}
\end{eqnarray}
		%	We claim that $\{\ctj\}$ is bounded in $W^{1,2}(\Omega_T)$.
		%	To this end, we calculate
	We are ready to estimate, with the aid of Lemma \ref{l34}, that
			\begin{eqnarray}
				\iot\left(\pt\rtj\right)^2dxdt&=&\sum_{k=1}^j\int_{t_{k-1}}^{t_k}\io\left(\frac{\esk-\esko}{\tau}\right)^2dxdt\nonumber\\
				&=&\sum_{k=1}^j\int_{t_{k-1}}^{t_k}\io(\sqrt{\esk}+\sqrt{\esko})^2\left(\frac{\sqrt{\esk}-\sqrt{\esko}}{\tau}\right)^2dxdt\nonumber\\
				&\leq &4	\|\rbj\|_{\infty,\Omega_T}\iot\left(\pt\ctj\right)^2dxdt\leq c.
				\end{eqnarray}
		As for the gradient with respect to the space variables, we have
				\begin{eqnarray}
				\int_{\Omega_T}|\nabla\rtj|^2 \, dxdt &=&
				\sum_{k=1}^j\int_{t_{k-1}}^{t_k}
				\int_{\Omega} \left|\frac{t-t_{k-1}}{\tau}\nabla \esk+\left(1-\frac{t-t_{k-1}}{\tau}\right)\nabla \esko\right|^2dxdt\nonumber\\
				&\leq&\sum_{k=1}^j\int_{t_{k-1}}^{t_{k-1}}\left[ \frac{t-t_{k-1}}{\tau}\int_{\Omega}\left|\nabla \esk \right|^2dx+\left(1-\frac{t-t_{k-1}}{\tau}\right)\int_{\Omega}\left|\nabla \esko\right|^2dx\right]dt\nonumber\\
				&=&\sum_{k=1}^j\tau\left( \int_{\Omega}\left|\nabla\esk\right|^2dx+\int_{\Omega}\left|\nabla\esko \right|^2dx\right)\nonumber\\
				%	&\leq&\sum_{k=1}^jc\tau\left[ \int_{\Omega}\left|\nabla^2 %\esk\right|^2dx+\int_{\Omega}\left|\nabla^2 \esko\right|^2dx\right]\nonumber\\
				&\leq& 2\int_{\Omega_T}\left|\nabla \rbj \right|^2dxdt+\tau\int_{\Omega}\left|\nabla \rho_0\right|^2dx	\leq c.\label{rg3}
			\end{eqnarray}
		The last step is due to Lemma \ref{l31}. The proof is complete.	
			\end{proof}	
	%	This, alone with Lemma \ref{l34}, gives the desired result.
			It follows that $\{\rtj\}$ is precompact in $L^2(\Omega_T)$. Note that
		\begin{eqnarray}
			\int_{\Omega_T}|\rtj-\rbj|^2dxdt &=&\sum_{k=1}^{j}\int_{t_{k-1}}^{t_k}(t_k-t)^2\int_{\Omega}\left(\frac{\esk-\esko}{\tau}\right)^2dxdt
			\nonumber\\
			&=&\sum_{k=1}^{j}\tau^3\int_{\Omega}\left(\pt\rtj\right)^2dx\nonumber\\
			&=&\tau^2\int_{\Omega_T}\left(\pt\rtj\right)^2dxdt\leq c\tau^2.\nonumber
		\end{eqnarray}
		Subsequently, 
		$\{\rbj\}$ is also precompact in $L^2(\Omega_T)$. As a result, we can select a subsequence of $\{\rbj\}$, still denoted by $\{\rbj\}$, such that
		%with the property
		\begin{equation}
			\mbox{$\rbj$ converges a.e. on $\Omega_T$.}\nonumber
		\end{equation}
	%	Thus $\rbj=\left(\srbj\right)^2$ also converges a.e. on $\omt$. 
		%In view of (H2) and Lemma \ref{wbd}, we have
	%This together with (R1)-(R3) implies  that $\{\rbj\}$ is precompact in $L^2(\Omega_T)$. 
		%For this purpose, 
		
		We define a function $\ftj(x,t)$ on $\Omega_T$ as follows: For each  $(x,t)\in\Omega_T$ there is a $k\in\{1,2,\cdots,j\}$ such that $t\in(t_{k-1}, t_k]$. Subsequently, we set
	%Introduce the following function
	%If $x\in\Omega,  \ t\in(t_{k-1},t_k]$, define the functions
	\begin{eqnarray}
	%	\wtj(x,t)&=&\frac{t-t_{k-1}}{\tau}\wk(x)+\left(1-\frac{t-t_{k-1}}{\tau}\right)w_{k-1}(x),  \nonumber\\
	%	\wbj(x,t)&=&\wk(x)
		%,\\
			\ftj(x,t)
			 &=&\frac{t-t_{k-1}}{\tau}(G_k(x)+\tau\ln\esk(x))
		%\lntj(x,t)
		%&=&\frac{t-t_{k-1}}{\tau}\ln\esk(x)
		+\left(1-\frac{t-t_{k-1}}{\tau}\right)(G_{k-1}(x)+\tau\ln\esko(x)).\nonumber
	\end{eqnarray}
%	for $x\in\Omega,  \ t\in(t_{k-1},t_k]$, and $k\in\{1,2,\cdots, j\}$. 
We can write \eqref{ss31} as
	\begin{equation}
		-\Delta\rtj=-\ftj\ \ \mbox{in $\Omega_T$.}\label{wj1}
	\end{equation}
Of course, we also have the boundary condition
% and \eqref{wk2} as
	\begin{eqnarray}
	%	-\Delta \wbj+2\wbj^{-1}|\nabla \wbj|^2&= &(\gbj+\tau\ln\rbj)\wbj^2 \ \ \mbox{in %$\Omega_T$,}\label{wj1}\\
		\nabla \rtj\cdot\nu&=&0\ \ \mbox{on $\po\times(0,T)$.}\label{wj2}
	\end{eqnarray}	
Set
	$$\wtj=\frac{1}{\rtj}.$$
	We can infer from \eqref{rinfty} that $\rtj$ is bounded away from $0$ below for each fixed $j$. Thus $\wtj$ is well-defined.
	%We can easily verify from \eqref{ss31} and \eqref{s33} 
	An elementary calculation from \eqref{wj1} and \eqref{wj2} shows that $\wtj$ satisfies the problem
	\begin{eqnarray}
		-\Delta \wtj+2\wtj^{-1}|\nabla \wtj|^2&=&\ftj\wtj^2\ \ \mbox{in $\Omega$,}\nonumber\\
		\nabla \wtj\cdot\nu&=&0\ \ \mbox{on $\po$}\nonumber
	\end{eqnarray}
for each $t\in [0,T]$. 
\begin{lem} 
		The sequence $\{\wtj\}$ is bounded in $L^\infty(\Omega_T)$.
\end{lem}
\begin{proof}
The proof largely mimics what we did in Subsection \ref{sub2}. First,
%Recall that $N$ is either $2$ or $3$. 
%To simplify our subsequent presentation, we set
we can deduce from \eqref{rg} and Lemmas \ref{l31} and \ref{l34} that
\begin{eqnarray}
\lefteqn{	\sup_{0\leq t\leq T}\|\ln\rbj(\cdot,t)\|_{1, \Omega}+\sup_{0\leq t\leq T}\|\ftj(\cdot,t)\|_{2, \Omega}}\nonumber\\
	&\leq &\sup_{0\leq t\leq T}\|\ln\rbj(\cdot,t)\|_{1, \Omega}+2\sup_{0\leq t\leq T}\|\gbj(\cdot,t)\|_{2, \Omega}\nonumber\\
	&\leq &c(\varepsilon_0+c_1\tau).\nonumber
\end{eqnarray}
Here
 $c$ depends on $\Omega$ only, while $c_1$ depends on both $\|u_0\|_{1,\Omega}$ and $\Omega_T$.
Now pick a number $L$ from $(1,\infty)$. Note that $-\ln s$  is convex on $(0,1)$. Also, for each $t\in[0,T]$ there is a $k\in\{1,2,\cdots, j\}$ such that $t\in(t_{k-1}, t_k]$. With these in mind, we derive that
 \begin{eqnarray}
 	\left|\{\wtj>L\}\right|&=&\left|\left\{\rtj<\frac{1}{L}\right\}\right|\nonumber\\
 	&\leq &\frac{1}{\ln L}\int_{\{\rtj\leq \frac{1}{L}\}}|\ln\rtj|dx\nonumber\\
 	&\leq&\leq \frac{1}{\ln L}\int_{\{\rtj\leq \frac{1}{L}\}}\left(\frac{t-t_{k-1}}{\tau}|\ln\esk(x)|+\left(1-\frac{t-t_{k-1}}{\tau}\right)|\ln\esko(x)|\right)dx\nonumber\\
 	&\leq&\frac{\sup_{0\leq t\leq T}\|\ln\rbj(\cdot,t)\|_{1, \Omega}}{\ln L}\leq \frac{c(\varepsilon_0+c_1\tau)}{\ln L}.\nonumber
 \end{eqnarray}
This is the new version of \eqref{rsm}. For the remainder, 
%From here on, we can invoke the proof near the end of Subsection \ref{sub2}, starting with \eqref{xu4}. 
all we need to do is to substitute $\{w,\varepsilon_0, G\}$ in Subsection \ref{sub2} for $\{\wtj,\varepsilon_0+c_1\tau, \ftj\}$ here. Thus, the new version of \eqref{xu4}  is:   Lemma \ref{wbd} asserts that there is a positive number $c=c(\Omega)$ such that
\begin{eqnarray}
	\|\wtj\|_{\infty,\Omega}&\leq&c\|\wtj\|_{1,\Omega}+ c\|\ftj \wtj^{2}\|_{2,\Omega}\nonumber\\
	&\leq&	c\int_{\{\wtj> L\}}\wtj dx+c\int_{\{\wtj\leq L\}}\wtj dx+ c(\varepsilon_0+c_1\tau)^{}	\|\wtj\|_{\infty,\Omega}^{2}\nonumber\\
	%	&\leq&cy_0^{\frac{1}{4}}+c\|\pt uw\|_{2,\Omega}^{2}y_0^{\frac{1}{4}}.
	%	\nonumber\\
	%&\leq&+c(\varepsilon_0+c_1\tau)^{}\|w\|_{\infty,\Omega}^{2}\nonumber\\
	&\leq& c\|\wtj\|_{\infty,\Omega}\left|\left\{\wtj >L\right\}\right|+cL+c(\varepsilon_0+c_1\tau)\|\wtj\|_{\infty,\Omega}^{2}\nonumber\\
	&\leq &\frac{ c(\varepsilon_0+c_1\tau)\|\wtj\|_{\infty,\Omega}}{\ln L}+cL+c(\varepsilon_0+c_1\tau)\|\wtj\|_{\infty,\Omega}^{2}.\nonumber
\end{eqnarray}
Recall that $\rtj$ is piece-wise linear in the time variable $t$. Thus $\|\rtj(\cdot,t)\|_{\infty,\Omega}$ is a continuous function of $t$ for each fixed $j$.
As we mentioned earlier,  $\rtj$ is bounded away from $0$ below for each fixed $j$. We can conclude that $\|\wtj(\cdot,t)\|_{\infty,\Omega}$ is also a continuous function of $t$ for each fixed $j$.
This enables us to apply the proof of Lemma \ref{prop2.2}. Let $L_0,\ s_0,\  s_1,\    g$ be determined as in Subsection \ref{sub2}. 
%Upon doing so, we conclude that 
If
% $(\varepsilon_0+c_1\tau)$ is so small that
\begin{equation}\label{xu12}
	\varepsilon_0+c_1\tau< s_0,\ \ \|e^{\Delta u_0-\tau u_0}\|_{\infty,\Omega}<s_1
\end{equation}
%$$\varepsilon_0+c_1\tau< s_0$$
then
\begin{equation}\label{xu13}
	\|\wtj(\cdot,t)\|_{\infty,\Omega}\leq g(\varepsilon_0+c_1\tau, L_0) \leq g(\varepsilon_0+c_1, L_0)\ \ \mbox{for all $t>0$.}
\end{equation}
On account of \eqref{ez}, \eqref{xu12} holds for $j$ sufficiently large. Hence \eqref{xu13} also holds for $j$ sufficiently large.
This completes the proof of the lemma.
\end{proof} 

\begin{lem}
	The sequence $\{\utj\}$ is bounded in $W^{1,2}(\Omega_T)$.
\end{lem}
\begin{proof} First, we derive from Lemma \ref{l34} and \eqref{rg2} that
	$$\io\left(\pt \utj\right)^2dx\leq 2\io \gbj^2dx+2\io\left(\rbj-1\right)^2dx\leq c.$$
 Use $\ubj$ as a test function in \eqref{omm22} to get
\begin{equation}\label{xu6}
	\io\left|\nabla\ubj\right|^2dx+\tau\io\ubj^2dx=\io\ubj\ln\rbj dx.
\end{equation}  
By \eqref{ubb} and Poincar\'{e}'s inequality, we have
\begin{eqnarray}
	\io\ubj\ln\rbj dx&=&\io\left(\ubj-\frac{1}{|\Omega|}\io\ubj dx\right)\ln\rbj dx+\frac{1}{|\Omega|}\io\ubj dx\io\ln\rbj dx\nonumber\\
		&\leq &\varepsilon\io\left(\ubj-\frac{1}{|\Omega|}\io\ubj dx\right)^2dx+c(\varepsilon)\nonumber\\
		&\leq &c\varepsilon\io\left|\nabla\ubj\right|^2dx+c.\nonumber
		\end{eqnarray}
Use this in \eqref{xu6} to get
\begin{equation}
\sup_{0\leq t\leq T}	\io\left|\nabla\ubj\right|^2dx\leq c.\nonumber
\end{equation}	
Use Poincar\'{e}'s inequality again to derive
$$\io\ubj^2 dx\leq 2\io\left(\ubj-\frac{1}{|\Omega|}\io\ubj dx\right)^2dx+\frac{2}{|\Omega|}\left(\io\ubj dx\right)^2\leq c.$$
By a calculation similar to \eqref{rg3}, we obtain
$$\|\utj\|_{W^{1,2}(\Omega)}\leq c.$$
The proof is complete.\end{proof}

To finish the proof Theorem \ref{thm1},
we square both sides of \eqref{omm22} and integrate to get
$$\io\left(\Delta\ubj\right)^2dx+2\tau\io\left|\nabla\ubj\right|^2dx+\tau^2\io\ubj^2=\io\ln^2\rbj dx\leq c.$$
This together with \eqref{2.2} implies
$$\sup_{0\leq t\leq T}\|\ubj\|_{W^{2,2}(\Omega)}\leq c.$$
%The last step is due to \eqref{xu13}. 
Similarly,
$$\sup_{0\leq t\leq T}\|\rbj\|_{W^{2,2}(\Omega)}\leq c.$$
%Obviously, $\{\ubj\}$ is bounded in $L^\infty(\Omega)$, and so is $\{\utj\}$.
 We are ready to pass to the limit in the system \eqref{omm1}-\eqref{omm22}. This completes the proof of Theorem \ref{thm1}.
			
  \bigskip
  
  %\noindent{\bf Acknowledgment:}  This work was completed while the second author
 % was visiting Duke University. He would like to express his gratitude for the hospitality of
 % the hosting institute and the financial support from the KI-Net for his visit. The research of JL was partially supported by
 % KI-Net NSF RNMS grant No. 1107291 and NSF grant DMS 1514826.	

\end{document}